\documentclass[reqno]{amsart}

\usepackage[english]{babel}
\usepackage[letterpaper,top=2cm,bottom=2cm,left=3cm,right=3cm,marginparwidth=1.75cm]{geometry}
\usepackage{amsmath, amssymb}
\usepackage{amsthm,xcolor}
\usepackage{graphicx}
\usepackage{enumitem}
\usepackage{multicol}
\usepackage[colorlinks=true, allcolors=blue]{hyperref}
\usepackage{tikz}
\usepackage{appendix}

\usepackage{cleveref}
\usepackage{etoolbox}
\newcounter{taggedeq}
\pretocmd{\equation}{\stepcounter{taggedeq}}{}{}

\allowdisplaybreaks

\newcommand{\df}[1]{\textbf{#1}}

\newcommand{\dha}{d_H}

\newtheorem{theorem}{Theorem}[section]
\newtheorem{proposition}[theorem]{Proposition}
\newtheorem{corollary}[theorem]{Corollary}
\newtheorem{lemma}[theorem]{Lemma}

\theoremstyle{definition}
\newtheorem{definition}[theorem]{Definition}
\newtheorem{remark}[theorem]{Remark}

\numberwithin{equation}{section}

\title[Generalized \texorpdfstring{$\varphi$}{varphi}-pullback attractors]{Generalized \texorpdfstring{$\varphi$}{varphi}-pullback attractors for evolution processes and application to a nonautonomous wave equation}
\author[M.C. Bortolan]{Matheus C. Bortolan}
\author[T. Caraballo]{Tom\'as Caraballo}
\author[C. Pecorari Neto]{Carlos Pecorari Neto}

\begin{document}

\makeatletter\def\Hy@Warning#1{}\makeatother
\let\thefootnote\relax\footnotetext{
[M.C. Bortolan] This study was financed in part by the Coordena\c c\~ao de Aperfei\c coamento de Pessoal de N\'ivel Superior - Brasil (CAPES) - Finance Code 001 and by the Ministerio de Ciencia e Innovaci\'on, Agencia Estatal de Investigaci\'on (AEI) and FEDER under project PID2021-122991NB-C21.
\par \quad Address: Departamento de Matem\'atica, Centro de Ci\^encias F\'isicas e Matem\'aticas, Universidade Federal de Santa Catarina (UFSC), Campus Florian\'opolis, CEP 88040-090, Florian\'opolis - SC, Brasil.

e-mail: \texttt{m.bortolan@ufsc.br}, ORCiD: 0000-0002-4838-8847.

[T. Caraballo] This study was financed in part by Ministerio de Ciencia e Innovaci{\'o}n (MCI), Agencia Estatal de Investigaci{\'o}n (AEI) and Fondo Europeo de Desarrollo Regional (FEDER) under the project PID2021-122991NB-C21.
\par \quad Address: Departamento de Ecuaciones Diferenciales y An\'alisis Num\'erico, Universidad de Sevilla, Campus Reina Mercedes, 41012, Seville, Spain. e-mail: \texttt{caraball@us.es}, ORCiD: 0000-0003-4697-898X.

[C. Pecorari Neto] This study was financed in part by the Coordena\c c\~ao de Aperfei\c coamento de Pessoal de N\'ivel Superior - Brasil (CAPES) - Finance Code 001.
\par \quad Address: Departamento de Matem\'atica, Centro de Ci\^encias F\'isicas e Matem\'aticas, Universidade Fe\-de\-ral de Santa Catarina (UFSC), Campus Florian\'opolis, CEP 88040-090, Florian\'opolis - SC, Brasil. 

e-mail: \texttt{carlospecorarineto@gmail.com}, ORCiD: 0009-0007-4218-7781.
}

\maketitle

\begin{abstract}
In this work we define the \textit{generalized $\varphi$-pullback attractors} for evolution processes in complete metric spaces, which are compact and positively invariant families, such that they \textit{pullback attract} bounded sets with a \text{rate} determined by a decreasing function $\varphi$ that vanishes at infinity. We find conditions under which a given evolution process has a generalized $\varphi$-pullback attractor, both in the discrete and in the continuous cases. We present a result for the special case of generalized polynomial pullback attractors, and apply it to obtain such an object for a nonautonomous wave equation.
		
\bigskip \noindent \textbf{Keywords}: Generalized $\varphi$-pullback attractors. $\varphi$-pullback $\kappa$-dissipativity. Nonautonomous wave equation. Evolution processes. Pullback attractors.

\bigskip \noindent \textbf{MSC2020:} 35B41, 35L20, 37L25.
\end{abstract}

\section{Introduction}

The study of nonautonomous problems has been the focus of many researchers in the last decades, such as \cite{Bortolan2014,Caraballo2006,CRLBook,KloRas2011,Sell1967a,Sell1967b}, for example. Nonautonomous equations and systems are of fundamental importance to model and understand real problems in Chemistry, Biology, Physics, Economics and many other areas, since naturally, the independent terms that usually appear in the models are time-dependent forces.

It is also of great importance to deal with the asymptotic behavior of such problems and, for that, the key objects are the \textit{attracting sets} and \textit{families}. The well recognized existing literature deals with attractors that, in general terms, represent the sets of limit states of the trajectories of the equations, and also contain all the bounded solutions defined for all time. Thus, roughly speaking, the attractors are the objects that contain the ``relevant'' solutions of the systems, bearing in mind real world problems. 

To be more precise about our goals in this work, we begin by presenting an overview of the theory of evolution processes and their pullback attractors in metric spaces. In what follows we write $\mathbb{T}$ to denote either the set of the real numbers ($\mathbb{R}$) or the integers ($\mathbb{Z}$). We will also denote $\mathbb{T}^+:=\{t\in \mathbb{T}\colon t\geqslant 0\}$.

By setting $\mathcal{P}=\{(t,s)\in \mathbb{T}^2\colon t\geqslant s\}$ and considering $(X,d)$ a metric space, we say that a two-parameter family $S=\{S(t,s)\colon (t,s)\in \mathcal{P}\}$ of continuous maps from $X$ into itself is an \df{evolution process}\index{evolution process} if
\begin{enumerate}[label=$\circ$]
\item $S(t,t)x=x$ for all $x\in X$ and $t\in \mathbb{T}$;
\item $S(t,r)S(r,s)=S(t,s)$ for all $(t,r), (r,s)\in \mathcal{P}$, that is, $t,r,s\in \mathbb{T}$ and $t\geqslant r \geqslant s$;
\item the map $\mathcal{P}\times X \ni (t,s,x) \mapsto S(t,s)x\in X$ is continuous.
\end{enumerate}
The evolution process $S$ is said to be a \df{continuous} \df{evolution process}\index{evolution process!continuous} if $\mathbb{T}=\mathbb{R}$ and a \df{discrete evolution process}\index{evolution process!discrete} if $\mathbb{T}=\mathbb{Z}$. We point out that, by definition, all evolution processes have continuity properties. The name \textit{continuous evolution process} refers only to the time parameter being taken in $\mathbb{R}$, to separate them from their discrete counterpart.

The setting of evolution processes is what we call the \textit{nonautonomous framework}, due to the explicit dependence of both initial time $s$ and final time $t$, separately. When there exists a family $T=\{T(t)\colon t\in \mathbb{T}^+\}$ such that $S(t,s)=T(t-s)$ for all $(t,s)\in \mathcal{P}$, we say that $T$ is a \df{semigroup}\index{semigroup} and that $S$ is an \textit{autonomous evolution process}. This indicates that $S$ depends only on the elapsed time $t-s$ and not explicitly on $t$ and $s$ separately.

To mathematically introduce the \textit{pullback attraction theory}, which is the one we develop in this paper, we let $\mathfrak{F}$\index{class of families}\index{family} be the class of all families $\hat{D}=\left\{D_t\right\}_{t\in\mathbb{T}}$, where $D_t$ is a nonempty subset of $X$ for each $t\in\mathbb{T}$. A family $\hat{A}\in \mathfrak{F}$ is said to be \df{closed/compact}\index{family!closed}\index{family!compact} if $A_t$ is a closed/compact subset of $X$ for each $t\in \mathbb{T}$. Moreover, given $\hat{A},\hat{B}\in\mathfrak{F}$ we say that $\hat{A}\subset \hat{B}$ if $A_t\subset B_t$ for each $t\in\mathbb{T}$. 

For an evolution process $S$ in $X$, we say that a family $\hat{B}\in \mathfrak{F}$ is \df{pullback attracting} if for each bounded set $D\subset X$ and $t\in \mathbb{T}$ we have
\[
\lim_{s\to  -\infty}\dha (S(t,s)D,B_t)=0,
\]
where
$\dha (U,V)=\sup_{u\in U}\inf_{v\in V}d(u,v)$
denotes the \textbf{Hausdorff semidistance}\index{Hausdorff semidistance} between two nonempty subsets $U$ and $V$ of $X$. The term \textit{pullback attraction} emphasizes that we are fixing a final time $t$ and taking the initial time $s\to -\infty$. Hence, we are \textit{pulling back} to the present time the solutions which are starting further in the past. 

Consider a given evolution process $S$ in $X$. We say that a family $\hat{B}$ is \df{invariant/positively invariant}\index{family!invariant}\index{family!positively invariant} for $S$ if for all $(t,s)\in \mathcal{P}$ we have $S(t,s)B_s = B_t$ / $S(t,s)B_s\subset B_t$, respectively. With these definitions, we can present the notion of a \textit{pullback attractor}:f or an evolution process $S$ in a metric space $X$, we say that $\hat{A}\in\mathfrak{F}$ is a \df{pullback attractor} for $S$ if:
\begin{enumerate}[label=(\roman*)]
\item $\hat{A}$ is compact;
\item $\hat{A}$ is invariant for $S$;
\item \label{ppa3} $\hat{A}$ is a pullback attracting family;
\item \label{ppa4} $\hat{A}$ is the minimal closed set satisfying \ref{ppa3}, that is, if $\hat{C}\in\mathfrak{F}$ is a pullback attracting closed family, then $\hat{A}\subset\hat{C}$.
\end{enumerate}

The minimality condition \ref{ppa4} is there to ensure that when a pullback attractor for $S$ exists, it is unique. This condition could be replaced by a number of other properties such as, for instance, that for each $t\in \mathbb{T}$ the set $\cup_{s\leqslant t}A_s$ is bounded in $X$. When a family satisfies this property, we say that it is a \textbf{backwards bounded family}.\index{family!backwards bounded}

The problem that we face is the following: assuming that a pullback attractor $\hat{A}$ exists, there is no qualitative information regarding the \textit{rate of attraction} of $\hat{A}$. To that end, many authors (for instance, \cite{Sonner1,Sonner2}) have worked with the notion of a  \textit{pullback exponential attractor}, that is, a compact family $\hat{M}\in  \mathfrak{F}$ which is positively invariant for $S$ and \textit{exponentially pullback attracts} bounded subsets of $X$, that is, there exists a positive constant $\omega>0$ for which for all $D\subset X$ bounded we have
\[
\lim_{s\to -\infty} e^{\omega s}\dha(S(t,s)D,M_t) = 0,
\]
and, furthermore, the \textit{fractal dimension}\index{fractal dimension} of $M_t$ is uniformly bounded for $t\in \mathbb{T}$, that is, there exists $c_0\geqslant 0$ such that for all $t\in \mathbb{T}$ we have
\[
\dim_F(M_t):= \lim_{r\to 0^+} \frac{\ln N(r,M_t)}{\ln\frac1r} \leqslant c_0,
\]
where $N(r,M_t)$ denotes the minimum number of balls of radius $r>0$ in $X$ that cover $M_t$. The fractal dimension is an upper bound to both the \textit{Hausdorff dimension} and the \textit{topological dimension} of $K$. One of its main use is the following: if $\dim_F(K)$ is finite, then $K$ can be projected injectively in a finite dimensional space, with dimension larger then $2\dim_F(K)+1$ (see \cite{Mane1980}).

In the work done by Zhao, Zhong and Zhu in \cite{Zhao2022}, the authors introduce the concept of \textit{compact $\varphi$-attracting sets} in the autonomous framework (that is, for semigroups), which is a fixed set $M$ that contains the global attractor of the system and attracts bounded sets with rate of attraction given by a function $\varphi$, which is not necessarily the exponential function and with no necessary bound for its fractal dimension. More recently, in \cite{yan2023long}, the authors studied a nonautonomous version of this theory, using the Shatah–Struwe solutions, but focusing their efforts in the \textit{uniform attractor}, which is a different framework than the one of pullback attractors.

Inspired by those works, and combining the results with the ones of \cite{Zhang2017}, our goal is to expand this theory for the nonautonomous case, in the framework of pullback attraction, and understand the asymptotic behavior of nonautonomous problems (by means of evolution processes) defining the \textit{generalized $\varphi$-pullback attracting families}, which we present in what follows. To begin, we will define the functions that will determine  the \textit{decay rate} of the pullback attraction, and we chose to name them \textit{decay functions}.

\begin{definition}[Decay function]\label{def:DecayFunction}
We say that a function $\varphi\colon [k,\infty) \to [0,\infty)$, where $k\geqslant 0$ is an appropriate constant, is a \df{decay function}\index{decay function} if $\varphi$ is decreasing, $\lim\limits_{t\to \infty}\varphi(t)=0$ and 
\begin{equation}\label{eq:DecayCondition}
\limsup\limits_{t \to \infty}\dfrac{\varphi(\omega t + \eta)}{\varphi(\omega t)}<\infty \quad \hbox{ for every } \omega>0 \hbox{ and } \eta\in \mathbb{R}.
\end{equation}
\end{definition}

With that, given an evolution process $S$ in a metric space $X$, we can define the main object of our study.

\begin{definition}[Generalized \texorpdfstring{$\varphi$}{varphi}-pullback attractor]\label{def:GenPA}
We say that a family $\hat{M}\in \mathfrak{F}$ is a \df{generalized $\varphi$-pullback attractor}\index{generalized $\varphi$-pullback attractor} for $S$ if $\hat{M}$ is compact, positively invariant and \df{$\varphi$-pullback attracting}\index{$\varphi$-pullback attraction}, that is, there exists a constant $\omega>0$ such that for \textit{every} bounded $D\subset X$ and $t\in\mathbb{T}$ there exist $C=C(D,t)\geqslant 0$ and $\tau_0=\tau_0(D,t)\geqslant 0$ such that 
\[
\dha (S(t,t-\tau)D,M_t)\leqslant C\varphi(\omega\tau) \quad \hbox{ for all } \tau\geqslant \tau_0.
\]
\end{definition}
Note that the parameter $\omega>0$ is fixed, and works for all bounded subsets $D$ of $X$ and $t\in \mathbb{T}$, thus, the function $\varphi(\omega \cdot)$ is the one that controls the rate of pullback attraction of the family $\hat{M}$. 
We will work with the notion of generalized $\varphi$-pullback attractors, and whether their fractal dimension are uniformly bounded or not is, so far to our knowledge, an open question.

In Section \ref{Dphitheory}, we study what are the conditions these processes must satisfy to ensure the existence of such families. Inspired by \cite{Sonner1,Sonner2}, we divide our results in the \textit{discrete case} (see Theorem \ref{existfam}) and the \textit{continuous case} (see Theorem \ref{continuouscase}). Also, inspired by the theoretical work in \cite{ZhaoZhong2022}, we present a result specifically designed to prove the existence of a generalized \textit{polynomial} pullback attractor for a continuous evolution process (see Theorem \ref{corMain}). As expected from the usual theory of exponential pullback attractors, we prove that when a backwards bounded generalized $\varphi$-pullback attractor $\hat{M}$ for an evolution process $S$ exists, then $S$ also has a pullback attractor $\hat{A}$, with $\hat{A}\subset \hat{M}$ (see Theorem \ref{theo:GenimpliesPA}).

As our main result (see Theorem \ref{App:PolAtt} in Section \ref{Application}), also inspired by \cite{ZhaoZhong2022}, we apply the abstract theory of Section \ref{Dphitheory} to prove the existence of a generalized $\varphi$-pullback attractor, where $\varphi$ is a function with polynomial decay, for a class of nonautonomous wave equations given by:
\begin{equation}\label{ourproblem}\tag{NWE}
\left\{\begin{aligned}
& u_{tt}(t,x)-\Delta u(t,x)+k(t)\|u_t(t, \cdot)\|^p_{L^2(\Omega)}u_t(t,x)+f(t,u(t,x))\\
& \hspace{130pt} = \int_{\Omega}K(x,y)u_t(t,y)dy+h(x), (t,x)\in [s,\infty)\times \Omega,\\
& u(t,x) = 0, (t,x)\in \left[s,\infty \right) \times \partial\Omega,\\
& u(s,x)=u_0(x), \ u_t(s,x)=u_1(x), x\in \Omega,
\end{aligned} \right.
\end{equation}
where $\Omega\subset \mathbb{R}^3$ is a bounded domain with smooth boundary $\partial\Omega$. Specifically, the decay function will be $\varphi(t)=t^{-\frac1p}$, where $p>0$ is the one in \eqref{ourproblem}.

\section{Generalized \texorpdfstring{$\varphi$}{varphi}-pullback attractors} \label{Dphitheory}

Now we generalize the theory presented in \cite{Zhao2022} to the nonautonomous pullback framework, in a useful manner. At first sight, one might want to work with a map $\varphi(t,s)$ depending on \textbf{both} variables $t$ and $s$. This would imply that for each final time $t$ the map $\varphi(t,s)$ describes the decay of the solution as $s\to -\infty$. Since the behavior of $\varphi(t,s)$ can be different for distinct values of $t$ (for instance, it could be polynomial for one parameter $t$ and exponential for another), there would be no correct answer to the question: what is the rate of attraction? The answer would depend heavily on the final time $t$, and even if we could do that, it would not be practical for the applications. Thus, as defined in the introduction, we will work with single-variable functions, which we call \textit{decay functions}\index{decay function}. To recall the reader (see Definition \ref{def:DecayFunction}), a function $\varphi\colon [k,\infty) \to [0,\infty)$, where $k\geqslant 0$ is an appropriate constant, is called a decay function if $\varphi$ is decreasing, $\lim\limits_{t\to \infty}\varphi(t)=0$ and satisfies \eqref{eq:DecayCondition}. 

Examples of decay functions $\varphi$ are
\[
\varphi(t)=ce^{-\beta t}, \quad \varphi(t)=ct^{-\beta}, \quad \hbox{ and } \quad \varphi(t)=c\ln^{-\beta}(t),
\]
with $c, \beta$ positive constants. Indeed, they are all decreasing and tend to zero at infinity. For \eqref{eq:DecayCondition} we note that
\begin{align*}
    \lim\limits_{t\to\infty}\frac{ce^{-\beta(\omega t+\eta)}}{ce^{-\beta(\omega t)}}=e^{-\beta \eta}, \ \ \lim\limits_{t\to\infty}\frac{c(\omega t + \eta)^{-\beta}}{c(\omega t)^{-\beta}}=1 \hbox{ and } \lim\limits_{t\to \infty}\frac{c\ln^{-\beta}(\omega t+\eta)}{c\ln^{-\beta}(\omega t)}=1.
\end{align*}
Condition \eqref{eq:DecayCondition} always occurs for a decreasing function $\varphi$ when $\eta\geqslant 0$, but that is not always the case when $\eta<0$. Indeed, the function $\varphi(t)=t^{-t}$, although decreasing and $0$ at infinity, does not satisfy \eqref{eq:DecayCondition} for $\eta<0$, since
\begin{align*}
\lim\limits_{t\to\infty}\frac{(\omega t+\eta)^{-\omega t -\eta}}{(\omega t)^{-\omega t}}=\lim\limits_{t\to \infty}\left(\frac{\omega t}{\omega t + \eta}\right)^{\omega t}(\omega t+\eta)^{-\eta} = \infty.
\end{align*}
We added \eqref{eq:DecayCondition} to avoid having to deal with translations of the function $\varphi$, and thus the rate of attraction will depend only on $\varphi$ and $\omega$, but not on $\eta$. We could have chosen to work without this assumption, but since it is satisfied by the usual decay functions (exponential, polynomial and logarithmic) we have decided to use it. It makes the notation a little easier on the eyes.

This definition of a generalized $\varphi$-pullback attractor (see Definition \ref{def:GenPA}), that is, a family $\hat{M}\in \mathfrak{F}$ which is compact, positively invariant and $\varphi$-pullback attracts all bounded subsets of $X$, is inspired by the one of an \textit{exponential attractor}, and the word \textit{generalized} is there to emphasize that we are \textit{not} asking hypotheses on the finitude of the fractal dimension. The question of what are the conditions required to obtain the bound for the fractal dimension of such object is still an open problem.

Firstly, we show that the theory of \textit{generalized} $\varphi$-pullback attractors, that is, without the bound of the fractal dimension, is meaningful only when dealing with problems without properties of compactness on the evolution processes. For a subset $A$ of $X$, we denote by $\overline{A}$ the closure of $A$ in $X$.

\begin{proposition}
Let $X$ be a metric space and $S$ an evolution process on $X$. Assume the following:
\begin{enumerate}[label=$\circ$]
    \item $S$ is an \df{eventually compact}\index{evolution process!eventually compact} evolution process, that is, there exists $\tau>0$ such that $S(t,t-\tau)$ is a compact map for each $t\in \mathbb{T}$;
    \item there exists a backwards bounded \textbf{pullback absorbing}\index{family!pullback absorbing} family $\hat{B}$, that is, given $D\subset X$ bounded and $t\in \mathbb{T}$ there exists $s_0\leqslant t$ such that $S(t,s)D\subset B_t$ for all $s\leqslant s_0$.
\end{enumerate}
Then, given any decay function $\varphi$, $S$ has a generalized $\varphi$-pullback attractor.
\end{proposition}
\begin{proof}
    Fix $t\in \mathbb{T}$ and set $B_\ast:=\cup_{s\leqslant t}B_s$, which is bounded by hypotheses. Since $\hat{B}$ is pullback absorbing, there exists $s_0:=s_0(t)\leqslant t$ such that $S(t,s)B_\ast \subset B_t$ for all $s\leqslant s_0$, which implies, in particular, that 
    $
        S(t,s)B_s \subset B_t \hbox{ for all } s\leqslant s_0.
    $ 
    It is clear that if $t_1\leqslant t_2$ we can choose $s_0(t_1)\leqslant s_0(t_2)$. We define the family $\hat{C}$ by
    \[
        C_t := \bigcup_{s\leqslant s_0(t)} S(t,s)B_s.
    \]
    Clearly $C_t\subset B_t$ and, hence, it is backwards bounded. Also, since for $s\leqslant t$ we have $s_0(s)\leqslant s_0(t)$, we have
    \[
        S(t,s)C_s = \bigcup_{r\leqslant s_0(s)}S(t,r)B_r \subset \bigcup_{r\leqslant s_0(t)}S(t,r)B_r = C_t,
    \]
    which proves that the family $\hat{C}$ is positively invariant. Furthermore, if $D\subset X$ is bounded then for each $s\in \mathbb{T}$ there exists $s_1=s_1(s)\leqslant s$ such that $S(s,r)D\subset B_s$ for all $r\leqslant s_1$. Thus if $s\leqslant s_1(s_0(t))$ then 
    \[
    S(t,s)D = S(t,s_0(t))S(s_0(t),s)D \subset  S(t,s_0(t)) B_{s_0(t)} \subset C_t, 
    \]
    which proves that $\hat{C}$ is pullback absorbing.

    Now we define the family $\hat{M}$ by
    $
    M_t = \overline{S(t,t-\tau)C_{t-\tau}} \hbox{ for each } t\in \mathbb{T}.
    $
    We claim that for any given decay function $\varphi$, $\hat{M}$ is a generalized $\varphi$-pullback attractor for $S$. Firstly we note that, since $C_{t-\tau}$ is bounded and $S(t,t-\tau)$ is compact, $M_t$ is compact. Now for $s\leqslant t$, since $S(t,s)$ is continuous from $X$ into $X$, we have
    \begin{align*}
    S(t,s)M_s & = S(t,s)\overline{S(s,s-\tau)C_{s-\tau}} \subset \overline{S(t,s-\tau)C_{s-\tau}} \\
    & = \overline{S(t,t-\tau)S(t-\tau,s-\tau)C_{s-\tau}} \subset \overline{S(t,t-\tau)C_{t-\tau}} = M_t,
    \end{align*}
    hence $\hat{M}$ is positively invariant. Lastly, if $D$ is bounded and $t\in \mathbb{R}$, since $\hat{C}$ is pullback absorbing there exists $s_0=s_0(t-\tau)$ such that $S(t-\tau,s)D\subset C_{t-\tau}$ for all $s\leqslant s_0$. Thus for $s\leqslant s_0$ we have
    \[
    S(t,s)D = S(t,t-\tau)S(t-\tau,s)D \subset S(t-\tau)C_{t-\tau} \subset M_t,
    \]
    which proves that $\hat{M}$ is pullback absorbing and, therefore, it is a generalized $\varphi$-pullback attractor for any given decay function $\varphi$. 
\end{proof}

This result shows that for either finite-dimensional problems or problems in infinite-dimensional spaces with compactness properties, the theory of generalized $\varphi$-pullback attractors, as is, is not that difficult, and the existence of a generalized $\varphi$-pullback attractor, for any given decay function $\varphi$, can be achieved by simply showing that $S$ has a backwards bounded pullback absorbing family $\hat{B}^0$.  Of course, if one adds the hypothesis of finitude of the fractal dimension of the family $\hat{M}$, the situation changes and turns the problem into a more difficult one. For now, as we already mentioned, we will not focus on the issue of the finitude of the fractal dimension.

In the literature, see \cite{Zhang2017} for instance, although not always explicit, there is a relationship between the existence of global attractors for semigroups and the decay of the Kuratowski measure of non-compactness for its $\omega$-limits. Hence, the work we present in what follows, inspired by \cite{Zhang2017}, attempts to unify the results of both \cite{Zhang2017} and \cite{Zhao2022} for the nonautonomous pullback setting. 

For what follows, unless clearly stated otherwise $(X,d) \hbox{ denotes a \textit{complete} metric space}$. Recall that for a nonempty bounded subset $C\subset X$, its \df{diameter}\index{diameter} is defined as $\operatorname{diam}(C) := \sup_{x,y\in C}d(x,y)$, and we have $\operatorname{diam}(C)=\operatorname{diam}(\overline{C})$. For $x_0\in X$ and $r>0$, the \df{open ball of radius $r$ centered in $x_0$}\index{ball!open} will be denoted by
$
B_r(x_0):= \{x\in X\colon d(x,x_0)<r\},
$
and the \df{closed ball of radius $r$ centered in $x_0$} \index{ball!closed} will be denoted by
$
\overline{B}_r(x_0):= \{x\in X\colon d(x,x_0)\leqslant r\}.
$
When there is a need to highlight the space $X$ in which the balls are being considered, we will use the notation $B_r^X(x_0)$ for open balls in $X$ and $\overline{B}_r^X(x_0)$ for closed balls in $X$. 

Recall that, for a bounded set $B\subset X$ we define its \df{Kuratowski measure} \df{of non-compactness}\index{measure!Kuratowski} by
\begin{align*}
\kappa(B)=\inf\{\delta > 0\colon B & \hbox{ admits a finite cover by sets of diameter less than or equal to }\delta\}.
\end{align*}
Related to the Kuratowski measure, we have the \df{ball measure} \df{of non-com\-pact\-ness}\index{measure!ball}, defined by
\[
\beta(B)=\inf\left\{r > 0\colon B \hbox{ admits a finite cover by open balls of radius } r\right\}.
\]
For results and properties of these measures, we refer to \cite{deimling2010nonlinear}. In the construction of generalized $\varphi$- pullback attractors, two concepts will play a fundamental role.

\begin{definition}[$\varphi$-pullback $\kappa$-dissipativity]\label{abcde} 
We say that an evolution process $S$ in $X$ is \df{$\varphi$-pullback $\kappa$-dissipative}\index{evolution process!$\varphi$-pullback $\kappa$-dissipative} if there exist $\omega>0$ such that for every bounded $D\subset X$ and $t\in\mathbb{T}$ there exists $C\geqslant 0$ and $\tau_0\geqslant 0$ such that
\[
\kappa\Big(\bigcup_{\sigma \geqslant \tau}S(t,t-\sigma)D\Big)\leqslant C\varphi(\omega\tau) \quad \hbox{ for all } \tau \geqslant \tau_0.
\]
\end{definition}

\begin{definition}[Uniform pullback absorption]\label{def:UnifPA}
Let $S$ be an evolution process on a complete metric space $X$. We say that $\hat{B}\in\mathfrak{F}$ is \textbf{uniformly pullback absorbing}\index{family!uniformly pullback absorbing} if given $D\subset X$ bounded and $t\in\mathbb{T}$, there exists $T>0$ such that $S(s,s-r)D\subset B_s$ for all $s\leqslant t$ and $r\geqslant T$.
\end{definition}

\subsection{The discrete case}

As mentioned in the introduction, inspired by \cite{Sonner1,Sonner2}, we first construct generalized $\varphi$-pullback attractors in the \textit{discrete case}, that is, when $S=\{S(n,m)\colon n\geqslant m \in \mathbb{Z}\}$ is a discrete evolution process on a complete metric space $X$. After that, we use the results of the discrete case to prove the existence of generalized $\varphi$-pullback attractors for the continuous case.

\begin{theorem}[Existence of generalized $\varphi$-pullback attractors for discrete evolution processes]\label{existfam}
Assume that there exists a closed family $\hat{B}=\left\{B_k\right\}_{k\in\mathbb{Z}}$ of bounded sets, which is uniformly pullback absorbing and positively invariant for the process $S$, and that, for a given decay function $\varphi$, $S$ is $\varphi$-pullback $\kappa$-dissipative. Then there exists a generalized $\varphi$-pullback attractor $\hat{M}$ for $S$, with $\hat{M}\subset \hat{B}$.
\end{theorem}
\begin{proof}
Let $\hat{B}$ as in the hypotheses. Since $S$ is $\varphi$-pullback $\kappa$-dissipative, there exists $\omega>0$ such that given $k\in\mathbb{Z}$ there exist $C\geqslant 0$, $m_0\geqslant 1$ and integer, where 
\[
\kappa\Big(\bigcup_{n\geqslant m}S(k,k-n)B_{k-n}\Big)< C\varphi(\omega m) \hbox{ for all }m\geqslant m_0.
\]
For $k\in \mathbb{Z}$ fixed
\begin{align*}
\beta\left(S(k,k-m_0)B_{k-m_0}\right)&\leqslant \kappa\left(S(k,k-m_0)B_{k-m_0}\right)\leqslant \kappa\Big(\bigcup_{n\geqslant m_0}S(k,k-n)B_{k-n}\Big)<C\varphi(\omega m_0),
\end{align*}
which means $S(k,k-m_0)B_{k-m_0}$ can be covered by a finite number of balls of radius $C\varphi(\omega m_0)$. Thus, there exist points $x_i^{(m_0)}\in X$ for $i=1,\ldots,r(m_0)$ such that 
\begin{align*}
S(k,k-m_0)B_{k-m_0}\subset \bigcup_{i=1}^{r(m_0)}B_{C\varphi(\omega m_0)}(x_i^{(m_0)})
\end{align*}
and, consequently, there exists $y_i^{(m_0)}\in B_{k-m_0}$  for $i=1,\ldots,r(m_0)$ such that
\begin{align*}
S(k,k-m_0)B_{k-m_0}\subset \bigcup_{i=1}^{r(m_0)}B_{2C\varphi(\omega m_0)}(z_i^{(m_0)}),
\end{align*}
where $z_i^{(m_0)}=S(k,k-m_0)y_i^{(m_0)}$. Analogously, there exists $y_i^{(m_0+1)}\in B_{k-(m_0+1)}$, for $i=1,\ldots,r(m_0+1)$, in such a manner that 
\begin{align*}
S(k,k-(m_0+1))B_{k-(m_0+1)}\subset \bigcup_{i=1}^{r(m_0+1)}B_{2C\varphi\left(\omega (m_0+1)\right)}(z_i^{(m_0+1)}),
\end{align*}
where $z_i^{(m_0+1)}=S(k,k-(m_0+1))y_i^{(m_0+1)}$, and so on.

Now, for each $k\in\mathbb{Z}$ and $n\in\mathbb{N}$, we define
\begin{align*}\mathcal{J}_k^n=\big\{S(k,k-(m_0+n))y_i^{(m_0+n)}\colon i=1,\ldots,r(m_0+n)\big\},
\end{align*}
recalling that the number $m_0$ and the points $y_j$ depend on the value of $k$. We also define $\mathcal{K}_k^0=\mathcal{J}_k^0$ and $\mathcal{K}_k^n=\mathcal{J}_k^n \cup S(k,k-1)\mathcal{K}_{k-1}^{n-1}$ for each $k\in\mathbb{Z}, n\in\mathbb{N}, n\geqslant 1$. These sets satisfy, for each $k\in\mathbb{Z}$ and $n, p\in\mathbb{N}$, the following properties (see Proposition \ref{technicaldetails} below):

\begin{enumerate}[label=(\roman*)]
\item $\mathcal{J}_k^n \subset S(k,k-(m_0+n))B_{k-(m_0+n)}\subset B_k$,\label{i1}
\item $ S(k,k-(m_0+n))B_{k-(m_0+n)}\subset \bigcup_{z\in \mathcal{J}_k^n}B_{2C\varphi(\omega(m_0+n))}(z)$,\label{i2}
\item $S(k,k-1)\mathcal{K}_{k-1}^n\subset \mathcal{K}_k^{n+1}$,\label{i3}
\item $S(k, k-(m_0+n))B_{k-(m_0+n)}\subset \bigcup_{z\in\mathcal{K}_k^n}B_{2C\varphi(\omega(m_0+n))}(z)$,\label{i4}
\item $\mathcal{K}_k^n\subset S(k,k-n)B_{k-n}\subset B_k$, \label{i5}
\item $S(k+p,k)\mathcal{K}_k^n\subset \mathcal{K}_{k+p}^{n+p}$,\label{i6}
\item $\mathcal{K}_k^n\subset S(k,k-m)B_{k-m}$ for each $m\in\mathbb{N}$ e $n\geqslant m$.\label{i7}
\end{enumerate} 
Lastly, we define the family $\hat{E}=\left\{E_k\right\}_{k\in\mathbb{Z}}$ by 
\[
E_k=\bigcup_{n=0}^{\infty} \mathcal{K}_k^n \quad \hbox{ for each integer } k\in \mathbb{Z}.
\] 
Now we prove that $\hat{E}$ is precompact, positively invariant, and $\varphi$-pullback attracting. For $m\in\mathbb{N}$, using \ref{i7} it follows that $\bigcup_{n=m+1}^{\infty}\mathcal{K}_k^n\subset S(k,k-m)B_{k-m}$. Thus, 
\begin{align*}
E_k=\bigcup_{n=0}^{\infty} \mathcal{K}_k^n = \bigcup_{n=0}^{m} \mathcal{K}_k^n \ \ \cup \bigcup_{n=m+1}^{\infty} \mathcal{K}_k^n \subset \Big(\bigcup_{n=0}^{m} \mathcal{K}_k^n\Big) \cup S(k,k-m)B_{k-m}.
\end{align*}
If $m\geqslant m_0$ we obtain
\begin{align*}
\kappa(E_k)&\leqslant \kappa\left(\Big(\bigcup_{n=0}^{m} \mathcal{K}_k^n\Big) \cup S(k,k-m)B_{k-m}\right) \leqslant \operatorname{max}\left\{\kappa\left(\bigcup_{n=0}^{m} \mathcal{K}_k^n\right), \ \kappa(S(k,k-m)B_{k-m})\right\}\\
&=\kappa(S(k,k-m)B_{k-m})\leqslant \kappa\left(\bigcup_{n\geqslant m}S(k,k-n)B_{k-n}\right)< C\varphi(\omega m ),
\end{align*}
and letting $m\to \infty$ we obtain $\kappa(E_k)=0$, which means that $E_k$ is precompact.

The positive invariance follows directly from \ref{i6}, since for $k\in\mathbb{Z}$ and $p\in\mathbb{N}$, we have
\begin{align*}
S(k+p,k)E_k = S(k+p,k)\bigcup_{n=0}^{\infty}\mathcal{K}_k^n = \bigcup_{n=0}^\infty S(k+p,k)\mathcal{K}_k^n \subset \bigcup_{n=0}^\infty \mathcal{K}_{k+p}^{n+p} \subset  \bigcup_{n=0}^\infty \mathcal{K}_{k+p}^{n}=E_{k+p}.
\end{align*}

For the $\varphi$-pullback attraction, fix $D\subset X$ bounded and $k\in\mathbb{Z}$. From the hypotheses there exists an integer $q\geqslant 1$ such that $S(r,r-n)D\subset B_r$ for all $r\leqslant k$ and $n\geqslant q$. If $n>q+m_0$ there exist $a,b\in\mathbb{N}$ such that $n=a+q$, $n=b+m_0$, $a>m_0$, $b>q$ and $a=m_0+b-q$. Now we have
\begin{align*}
\dha&(S(k,k-n)D, E_k)=\dha(S(k,k-a-q)D,E_k)=\dha\left(S(k,k-a)S(k-a,k-a-q)D,\bigcup_{j=0}^\infty \mathcal{K}_k^j\right)\\
&\leqslant \dha\left(S(k,k-a)B_{k-a},\bigcup_{j=0}^\infty \mathcal{K}_k^j\right)\leqslant \dha\left(S(k,k-a)B_{k-a}, \ \mathcal{K}_k^{b-q}\right)\\
&\leqslant 2C\varphi(\omega(m_0+b-q))=2C\varphi(\omega(n-q))=2C\varphi(\omega n -\omega q).
\end{align*}
where the last inequality follows from \ref{i4}, since
\begin{align*}
S(k,k-a)B_{k-a}&=S(k,k-(m_0+(b-q)))B_{k-(m_0+(b-q))}\subset \bigcup_{z\in\mathcal{K}_k^{b-q}}B_{2C\varphi(\omega (m_0+b-q))}(z).
\end{align*}
By the conditions imposed on $\varphi$ , there exists $C_1>0$ such that
\begin{align*}
    \dha&(S(k,k-n)D, E_k)\leqslant 2C\varphi(\omega n - \omega q)\leqslant C_1\varphi(\omega n)
\end{align*}
for $n$ bigger than a sufficiently large number $n_0$. This proves the $\varphi$-pullback attraction of $\hat{E}$.

Finally, we define the family $\hat{M}=\left\{M_k\right\}_{k\in\mathbb{Z}}$ by $M_k=\overline{E_k}$ for each $k\in \mathbb{Z}$. Since $M_k$ is the closure of a precompact set $E_k$, the compactness of each $M_k$ is obvious. Also, since $E_k\subset B_k$ and $\hat{B}$ is closed, we have $M_k\subset B_k$ for each $k\in \mathbb{Z}$, that is, $\hat{M}\subset \hat{B}$. For the positive invariance, let $k\in\mathbb{Z}$ and $p\in\mathbb{N}$. Then, by the continuity properties of $S$ and the positive invariance of $\hat{E}$,
\begin{align*}
S(k+p,k)M_k&=S(k+p,k)\overline{E_k} \subset \overline{S(k+p,k)E_k}\subset \overline{E_{k+p}}=M_{k+p}.
\end{align*}
Lastly, notice that if $n>n_0$, 
\begin{align*}
\dha(S(k,k-n)D,M_k)&=\dha(S(k,k-n)D,\overline{E_k})\leqslant \dha(S(k,k-n)D, E_k) \leqslant C_1\varphi(\omega n).
\end{align*}
This concludes the proof that $\hat{M}$ is a generalized $\varphi$-pullback attractor for $S$.
\end{proof}

To conclude the proof, we just have to show properties (i)-(vii).

\begin{proposition}\label{technicaldetails}
The sets $\mathcal{J}_k^n$ and $\mathcal{K}_k^n$ defined in the proof of the last theorem satisfy the properties (i)-(vii) therein.
\end{proposition}
\begin{proof}
(i). The result follows immediately from the definition of $\mathcal{J}_k^n$ and the positive invariance of $\hat{B}$.

\medskip (ii) By the construction we did previously, there exists $y_i^{(m_0+n)}\in B_{k-(m_0+n)}$, for $i=1,\ldots,r(m_0+n)$, such that 
\[
S(k,k-(m_0+n))B_{k-(m_0+n)}\subset \bigcup_{i=1}^{r(m_0+n)}B_{2C\varphi(\omega(m_0+n))}(z_i^{(m_0+n)}),
\]
where $z_i^{(m_0+n)}=S(k,k-(m_0+n))y_i^{(m_0+n)}$, for $i=1,..., r(m_0+n)$. Noting that $\mathcal{J}_k^n=\{z_i^{(m_0+n)}\colon i=1,\ldots,r(m_0+n)\}$, this item is also obvious.

\medskip (iii) For all $k\in\mathbb{Z}$ and $n\in\mathbb{N}$, it follows from the definition of sets $\mathcal{K}_k^n$ that 
\[
S(k,k-1)\mathcal{K}_{k-1}^n \subset \mathcal{J}_k^{n+1} \cup S(k,k-1)\mathcal{K}_{k-1}^n =\mathcal{K}_{k}^{n+1}.
\]

\medskip (iv) It follows from \ref{i2} and the fact that $\mathcal{J}_k^n\subset \mathcal{K}_k^n$ for all $k\in\mathbb{Z}$ and $n\in\mathbb{N}$.

\medskip (v) For each integer $k$, it follows from \ref{i1} that $\mathcal{K}_k^0=\mathcal{J}_k^0\subset B_k=S(k,k-0)B_{k-0}$. Using the positive invariance of $\hat{B}$ and item \ref{i1} again we also have 
\begin{align*}
\mathcal{K}_k^1&=\mathcal{J}_k^1 \cup S(k,k-1)\mathcal{K}_{k-1}^0=\mathcal{J}_k^1 \cup S(k,k-1)\mathcal{J}_{k-1}^0 \\
&\subset S(k,k-(m_0+1))B_{k-(m_0+1)}\cup S(k,k-1)B_{k-1}\\
&=S(k,k-1)S(k-1,k-m_0-1)B_{k-m_0-1} \cup S(k,k-1)B_{k-1}\subset S(k,k-1)B_{k-1}.
\end{align*}
Now let $k\in\mathbb{Z}$ fixed and suppose $n\in\mathbb{N}, n\geqslant 1$. By property \ref{i1}, positive invariance of $\hat{B}$ and the Principle of Induction on $n$,
\begin{align*}
\mathcal{K}_k^n&=\mathcal{J}_k^n \cup S(k,k-1)\mathcal{K}_{k-1}^{n-1}\\
& \subset S(k,k-m_0-n)B_{k-m_0-n} \cup S(k,k-1)S(k-1,k-1-(n-1))B_{k-1-(n-1)}\\
&=S(k,k-m_0-n)B_{k-m_0-n} \cup S(k,k-n)B_{k-n}\\
&=S(k,k-n)S(k-n,k-n-m_0)B_{k-n-m_0} \cup S(k,k-n)B_{k-n}\subset S(k,k-n)B_{k-n}.
\end{align*}

\medskip (vi) For $p=0$ it is clear that $S(k+0,k)\mathcal{K}_k^n=\mathcal{K}_{k+0}^{n+0}$ and for $p=1$ it is a immediate consequence of property \ref{i3}, indeed $S(k+1,k)\mathcal{K}_k^n\subset \mathcal{K}_{k+1}^{n+1}$. Now, by the Induction Principle on $p$, and using \ref{i3} again, we have
\begin{align*}
S(k+p+1,k)\mathcal{K}_k^n&=S(k+p+1,k+p)S(k+p,k)\mathcal{K}_k^n\subset S(k+p+1,k+p)\mathcal{K}_{k+p}^{n+p}\subset \mathcal{K}_{k+p+1}^{n+p+1}.
\end{align*}

\medskip (vii) Let $m\in\mathbb{N}$ fixed. If $n\geqslant m$,
\begin{align*}
\mathcal{K}_k^n\subset S(k,k-n)B_{k-n}=S(k,k-m)S(k-m,k-n)B_{k-n}\subset S(k,k-m)B_{k-m},
\end{align*}
where we have used property \ref{i5} and the positive invariance of $\hat{B}$.
\end{proof}

\subsection{The continuous case}

Using the proof of the discrete case we are able to the obtain a result for the continuous case.

\begin{theorem}[Existence of generalized $\varphi$-pullback attractors for continuous evolution processes]\label{continuouscase}
Let $S$ be a $\varphi$-pullback $\kappa$-dissipative continuous evolution process on $X$ and assume that there exists a closed family $\hat{B}=\left\{B_t\right\}_{t\in\mathbb{R}}$ of bounded sets, which is uniformly pullback absorbing and positively invariant. Suppose also that there exists $\gamma>0$ such that for $s\in \mathbb{R}$ and $0\leqslant \tau \leqslant \gamma$ there exists a constant $L_{\tau,s}>0$ for which 
\[
d(S(s+\tau,s)x,S(s+\tau,s)y)\leqslant L_{\tau,s}d(x,y) \quad \hbox{ for all } x,y\in B_s.
\]
Then there exists a generalized $\varphi$-pullback attractor $\hat{M}$ for $S$, with $\hat{M}\subset \hat{B}$.
\begin{proof}
Consider the discrete evolution process $S_d$ defined by 
\[
S_d(m,n)=S(\gamma m,\gamma n) \quad \hbox{ for } m,n\in\mathbb{Z} \hbox{ with } m\geqslant n.
\]
Then $S_d$ is a $\varphi$-pullback $\kappa$-dissipative discrete evolution process and there exists a uniformly pullback absorbing and positively invariant family $\hat{H}=\left\{H_k\right\}_{k\in\mathbb{Z}}$ where $H_k=B_{k\gamma}$ for each $k$. In the proof of Theorem \ref{existfam}, we verified that there exists a precompact, positively invariant and $\varphi$-pullback attracting family $\hat{E}=\left\{E_k\right\}_{k\in\mathbb{Z}}$ for the discrete evolution process $S_d$, with $\hat{E}\subset \hat{B}$.

Now, define the family $\hat{G}=\left\{G_t\right\}_{t\in\mathbb{R}}$ by $G_t=S(t,k\gamma)E_k$ for $t\in [k\gamma, (k+1)\gamma)$. Note that $G_{k\gamma}=E_k$ for each $k\in \mathbb{N}$ and $\hat{G}\subset \hat{B}$. We will prove that $\hat{G}$ is precompact, positively invariant and $\varphi$-pullback attracting for the process $S$.

The precompactness of each $G_t$ follows immediately from its definition, since $E_k$ is precompact and $S(t,k\gamma)$ is continuous. Now, let $t,s\in\mathbb{R}$, $t\geqslant s$. We have $s=p\gamma+p_1$, $t=q\gamma+q_1$ where $p,q\in\mathbb{Z}$, $q\geqslant p$ and $p_1,q_1\in [0,\gamma)$. Then,
\begin{align*}
S(t,s)G_s&=S(q\gamma+q_1,p\gamma+p_1)G_{p\gamma+p_1}=S(q\gamma+q_1,p\gamma+p_1)S(p\gamma+p_1,p\gamma)E_{p}\\
&=S(q\gamma+q_1,p\gamma)E_{p}=S(q\gamma+q_1,q\gamma)S(q\gamma,p\gamma)E_{p}\subset S(q\gamma+q_1,q\gamma)E_{q}=G_{q\gamma+q_1}=G_t,
\end{align*}
since $S(q\gamma,p\gamma)E_p=S_d(q,p)E_p\subset E_q$ by the positive invariance of $\left\{E_k\right\}_{k\in\mathbb{Z}}$ related to the discrete process $S_d$.

It remains to prove that $\hat{G}$ is $\varphi$-pullback attracting for $S$. Let $D\subset X$ bounded and $t\in\mathbb{R}$ (we can write $t=q\gamma+t_0$ where $q\in\mathbb{Z}$ and $t_0\in [0,\gamma[$). Taking the time $q\in\mathbb{Z}$, since $S_d$ is $\varphi$-pullback $\kappa$-dissipative, there exist $C\geqslant 0$ and $m_0\in\mathbb{N}_*$ such that $\kappa\left(\bigcup_{n\geqslant m}S_d(q,q-n)H_{q-n}\right)< C\varphi(\omega m)$ for all $m\geqslant m_0$, where $H_k=B_{k\gamma}$. Since $\hat{B}$ is uniformly pullback absorbing, there exists $T>0$ such that $S(\varsigma,\varsigma-r)D\subset B_\varsigma$ for all $\varsigma\leqslant t$ and $r\geqslant T$. Let $s\geqslant (m_0+2)\gamma+T$. It implies that $s\geqslant \gamma+T+t_0$ and, then, there exist $p\in\mathbb{N}$, $t_1\in [0,\gamma[$ such that $s=p\gamma+T+t_0+t_1$ and $p\geqslant m_0$. Now,
\begin{align*}
S(t,t-s)D&=S(t, (q-p)\gamma-T-t_1)D=S(t, q\gamma)S(q\gamma,(q-p)\gamma)S((q-p)\gamma,(q-p)\gamma-(T+t_1))D\\
&\subset S(t,q\gamma)S(q\gamma,(q-p)\gamma)B_{(q-p)\gamma},
\end{align*}
since $(q-p)\gamma \leqslant t$ and $T+t_1\geqslant t$, and thus
\begin{align}\label{qwer1}
\dha(S(t,t-s)D,G_t)&=\dha(S(t,t-s)D, S(t, q\gamma)E_q)\nonumber \leqslant \dha(S(t,q\gamma)S(q\gamma,(q-p)\gamma)B_{(q-p)\gamma},S(t,q\gamma)E_q)\nonumber\\
&\stackrel{(\ast)}{\leqslant} L \dha(S(q\gamma,(q-p)\gamma)B_{(q-p)\gamma}, E_q)=L \dha(S_d(q,q-p)H_{q-p}, E_q),
\end{align}
where in ($\ast$) we used the facts that $S(q\gamma,(q-p)\gamma)B_{(q-p)\gamma}\subset B_{q\gamma}$ and $E_q=\bigcup_{n=0}^\infty \mathcal{K}_q^n \subset H_q=B_{q\gamma}$, which are consequences of the positive invariance of $\hat{B}$ and item \ref{i5} of the proof of Theorem \ref{existfam}. Now, using item \ref{i4} from the proof of Theorem \ref{existfam}, since $p-m_0\geqslant 0$, we have
\begin{align*}
S_d(q,q-p)H_{q-p}&=S_d(q,q-(m_0+(p-m_0)))H_{q-(m_0+(p-m_0))}\subset \bigcup_{z\in \mathcal{K}_q^{p-m_0}}B^X_{2C\varphi(\omega p)}(z),
\end{align*}
where $C$ is an appropriate constant obtained from Theorem \ref{existfam}. This ensures that 
\begin{align}\label{qwer2}
\dha(S_d(q,q-p)H_{q-p}, \mathcal{K}_{q}^{p-m_0})\leqslant 2C\varphi(\omega p).
\end{align}
It follows from \eqref{qwer1} and \eqref{qwer2} that for $s\geqslant (m_0+2)\gamma+T$,
\begin{align*}
\dha& (S(t,t-s)D,G_t) \leqslant L \dha(S_d(q,q-p)H_{q-p}, E_q)=L \dha\Big(S_d(q,q-p)H_{q-p}, \bigcup_{n=0}^{\infty}\mathcal{K}_q^n\Big)\\
& \leqslant L \dha(S_d(q,q-p)H_{q-p},\mathcal{K}_q^{p-m_0})\leqslant 2LC\varphi(\omega p)=2LC\varphi\left(\omega\left(\tfrac{s-T-t_0-t_1}{\gamma}\right)\right)\\
&=2CL\varphi\left(\tfrac{\omega}{\gamma}s-\tfrac{\omega}{\gamma}T-\tfrac{\omega}{\gamma}(t_0+t_1)\right)\leqslant 2CL\varphi\left(\tfrac{\omega}{\gamma}s-\left[\tfrac{\omega}{\gamma}T+2\omega\right]\right).
\end{align*}
By the conditions imposed on $\varphi$, there exists $C_1>0$ such that $\dha(S(t,t-s)D,G_t)\leqslant C_1\varphi\left(\frac{\omega}{\gamma}s\right)$ for $s$ bigger than a sufficiently large number $s_0$.

Finally, define the family $\hat{M}=\left\{M_t\right\}_{t\in\mathbb{R}}$ by $M_t=\overline{G_t}$ for each $t\in \mathbb{R}$. Since $\hat{B}$ is closed and $\hat{G}\subset \hat{B}$, we have $\hat{M}\subset \hat{B}$. Furthermore, we have: $\hat{M}$ is compact, since $\hat{G}$ is precompact; $\hat{M}$ is positively invariant, since for all $t\geqslant s$ we know that 
\[
S(t,s)M_s=S(t,s)\overline{G_s}\subset \overline{S(t,s)G_s}\subset \overline{G_t}=M_t,
\]
and $\hat{M}$ is $\varphi$-pullback attracting for $S$, since for $s\geqslant s_0$,
\begin{align*}
\dha(S(t,t-s)D,M_t)&=\dha(S(t,t-s)D,\overline{G_t})\leqslant \dha(S(t,t-s)D,G_t) \leqslant C_1\varphi\big(\tfrac{\omega}{\gamma}s\big).
\end{align*}
\end{proof}
\end{theorem}

\subsection{Existence of a generalized polynomial pullback attractor}

The purpose of this subsection is to prove the following result, which will be used to prove the existence of a generalized polynomial pullback attractor for the nonautonomous wave equation \eqref{ourproblem} in Section \ref{Application}. To state this result, we state the following definition: we say that a family $\hat{B}$ is \textbf{uniformly bounded}\index{family!uniformly bounded} is $\cup_{t\in \mathbb{T}}B_t$ is a bounded subset of $X$.

We also present two definitions (also used in \cite{ZhaoZhong2022}) that will serve us for what follows. Let $X$ be a complete metric space and $B\subset X$. A function $\psi\colon X\times X \to \mathbb{R}^+$ is called \textbf{contractive}\index{contractive function} on $B$ if for each sequence $\left\{x_n\right\}_{n\in\mathbb{N}}\subset B$ we have 
\[
\liminf_{m,n\to \infty}\psi(x_n,x_m)=0.
\]
We denote the set of such functions by $\operatorname{contr}(B)$. Equivalently, a function $\psi$ is contractive on $B$ is for each sequence $\{x_n\}\subset B$ there exists a subsequence $\{x_{n_k}\}$ such that
\[
\lim_{k,\ell\to \infty}\psi(x_{n_k},x_{n_\ell})=0.
\]

Recall that a \textbf{pseudometric}\index{pseudometric} in a set $X$ is a function $\rho\colon X\times X\to [0,\infty)$ that satisfies:
\begin{enumerate}[label={$\circ$}]
    \item $\rho(x,x)=0$ for all $x\in X$;
    \item $\rho(x,y)=\rho(y,x)$ for all $x,y\in X$;
    \item $\rho(x,z) \leqslant \rho(x,y) + \rho(y,z)$ for all $x,y,z\in X$.
\end{enumerate}
Given $\varnothing \neq B\subset X$ and $\rho$ a pseudometric on $X$, we say that $\rho$ is \textbf{precompact on $B$}\index{pseudometric!precompact} if given $\delta>0$, there exists a finite set of points $\left\{x_1,...,x_r\right\}\subset B$ such that $B\subset \cup_{j=1}^r B_{\delta}^\rho(x_j)$,  where $B_{\delta}^{\rho}(x_j)=\{y\in X\colon \rho(y, x_j)<\delta\}.$
It is easy to see that $\rho$ is precompact on $B$ if and only if any sequence $\{x_n\}_{n\in \mathbb{N}}\subset B$ has a Cauchy subsequence $\{x_{n_j}\}_{j\in \mathbb{N}}$ with respect to $\rho$.

\begin{theorem}[Existence of generalized polynomial pullback attractors for continuous evolution processes]\label{corMain}
Let $X$ be a complete metric space and $S$ be a continuous evolution process in $X$ such that there exists a closed, uniformly bounded, and positively invariant uniformly pullback absorbing family $\hat{B}$ for $S$. Suppose that there exists $\gamma>0$ such that for each $s\in \mathbb{R}$ and $0\leqslant \tau \leqslant \gamma$ there exists a constant $L_{\tau,s}>0$ such that
\[
d(S(s+\tau,s)x,S(s+\tau,s)y)\leqslant L_{\tau,s}d(x,y) \hbox{ for all } x,y\in B_s.
\]
Assume also that there exist $\beta\in \left(0, 1\right)$, $r>0$, $T>0$, $C>0$ satisfying: given $t\in \mathbb{R}$, there exist functions $g_1, g_2\colon (\mathbb{R}^+)^m\to \mathbb{R}^+$, $\psi_1, \psi_2\colon X\times X \to \mathbb{R}^+$ and pseudometrics $\rho_1,\ldots,\rho_m$ on $X$ such that:
\begin{enumerate}[leftmargin=*,label={(\roman*)}]
\item $g_i$ is non-decreasing with respect to each variable, $g_i(0,...,0)=0$ and it is continuous at $(0,...,0)$ for $i=1,2$;
\item For each $n\in\mathbb{N}$, $\rho_1,\ldots,\rho_m$ are precompact on $B_{t-nT}$;
\item $\psi_1,\psi_2\in \operatorname{contr}(B_{t-nT})$ for all $n\in\mathbb{N}$;
\item for each $n\in\mathbb{N}$ and all $x,y\in B_{t-nT}$ we have
\[
d(S_nx,S_ny)^r \leqslant d(x,y)^r+g_1(\rho_1(x,y),...,\rho_m(x,y))+\psi_1(x,y);
\]
and
\begin{align*}
d(&S_nx,S_ny)^r \leqslant C \big[d(x,y)^r  - d(S_nx,S_ny)^r\\
& +g_1(\rho_1(x,y),...,\rho_m(x,y))+\psi_1(x,y)\big]^\beta + g_2(\rho_1(x,y),...,\rho_m(x,y))+\psi_2(x,y),
\end{align*}
where $S_n:=S(t-(n-1)T,t-nT)$ for each $n\in \mathbb{N}$.
\end{enumerate}
Then $S$ is  $\varphi$-pullback $\kappa$-dissipative, with the decay function $\varphi$ given by $\varphi(s)=s^{\frac{\beta}{r(\beta-1)}}$. Also, $S$ has a uniformly bounded generalized $\varphi$-pullback attractor $\hat{M}$.
\end{theorem}

To prove this theorem, we first present technical results.

\begin{proposition}\label{lemmaxnz}
Let  $\hat{B}\in\mathfrak{F}$ be a uniformly pullback absorbing family. Suppose that there exist a decay function $\varphi$ and $\omega>0$ such that for each $t\in\mathbb{R}$ there exist $C\geqslant 0$, $\tau_0>0$ such that 
\[
\kappa(S(t,t-\tau)B_{t-\tau})\leqslant C\varphi(\omega\tau) \quad \hbox{ for all } \tau\geqslant \tau_0.
\]
Then $S$ is $\varphi$-pullback $\kappa$-dissipative.
\begin{proof}
Let $D \subset X$ bounded and $t\in\mathbb{R}$. Since $\hat{B}$ is uniformly pullback absorbing, there exists $T>0$ such that $S(s,s-r)D\subset B_s$ for all $s\leqslant t$ and $r\geqslant T$. Take $\sigma \geqslant \tau_0+T>0$ and note that if $s\geqslant 2\sigma$, since $t-\sigma\leqslant t$ and $s-\sigma\geqslant T$, we have
\begin{align*}
S(t,t-s)D&=S(t,t-\sigma)S(t-\sigma,t-s)\\
&=S(t,t-\sigma)S(t-\sigma,(t-\sigma)-(s-\sigma))D\subset S(t,t-\sigma)B_{t-\sigma},
\end{align*}
which implies that $\bigcup_{s\geqslant 2\sigma}S(t,t-s)D \subset S(t,t-\sigma)B_{t-\sigma}$. Thus, since $\sigma\geqslant \tau_0$, it follows that
\begin{align*}
\kappa\Big(\bigcup_{s\geqslant 2\sigma}S(t,t-s)D\Big)\leqslant \kappa(S(t,t-\sigma)B_{t-\sigma})\leqslant C\varphi(\omega\sigma)
\end{align*}
for all $\sigma \geqslant T+\tau_0$. This is equivalent to 
\begin{align*}
\kappa\Big(\bigcup_{s\geqslant \tau}S(t,t-s)D\Big)\leqslant C\varphi\big(\tfrac{\omega}{2} \tau\big)
\end{align*}
for all $\tau \geqslant 2\tau_0+2T$, and the proof is complete.
\end{proof}
\end{proposition}

\begin{proposition}\label{thmplm}
Assume there exist $\beta\in \left(0, 1\right)$, $r>0$, $T>0$, $C>0$ and a positively invariant family $\hat{B}$ satisfying: given $t\in\mathbb{R}$, there exist $T_1>0$, $M>0$, functions $g_1, g_2\colon (\mathbb{R}^+)^m\to \mathbb{R}^+$, $\psi_1, \psi_2\colon X\times X \to \mathbb{R}^+$ and pseudometrics $\rho_1,\ldots,\rho_m$ on $X$ that satisfies the hypotheses (i)-(iv) of Theorem \ref{corMain}, and such that $\kappa(B_{t-s})\leqslant M$ for all $s\geqslant T_1$. Then there exists a constant $\omega=\omega(T,\beta, C)>0$ such that, given $t\in\mathbb{R}$, there exist constants $C_1>0$ and $s_1>0$ 
\[
\kappa\left(S(t,t-s)B_{t-s}\right)\leqslant C_1\left(\omega s\right)^{\frac{\beta}{r(\beta-1)}} \quad \hbox{ for all } s\geqslant s_1.
\]
\end{proposition}
\begin{proof}
Consider the function $u\colon \mathbb{R}^+\to \mathbb{R}^+$ defined by $u(s)=(3C)^{-1/\beta}s^{1/\beta}+s$. Since $u$ is an increasing bijective function, it has an inverse function that we will denote $v\colon \mathbb{R}^+\to \mathbb{R}^+$, which also is increasing. The composite functions $u^n$ and $v^n$ are also increasing for $n\geqslant 2$ and satisfy $u\leqslant u^2\leqslant u^3\leqslant \cdots$ and $v\geqslant v^2\geqslant v^3\geqslant \cdots$ for any $s\geqslant 0$. 

To simplify the notation even further, we set $D_n(x,y)=d(S_nx,S_ny)$, $G_1(x,y)=g_1(\rho_1(x,y),\cdots,$  $\rho_m(x,y))$ e $G_2(x,y)=g_2(\rho_1(x,y),\cdots,\rho_m(x,y))$. Then we have
\[
D_n(x,y)^r\leqslant d(x,y)^r+G_1(x,y)+\psi_1(x,y)
\]
and
\[
D_n(x,y)^r\leqslant C\left[d(x,y)^r-D(x,y)^r+G_1(x,y)+\psi_1(x,y)\right]^\beta+G_2(x,y)+\psi_2(x,y).
\]

Observe that 
\begin{align*}
u(&D_n(x,y)^r)=(3C)^{-\frac{1}{\beta}} D_n(x,y)^{\frac{r}{\beta}}+D_n(x,y)^r\\
&\leqslant(3C)^{-\frac{1}{\beta}}\left\{C\left[d(x,y)^r-D_n(x,y)^r+G_1(x,y)+\psi_1(x,y)\right]^\beta+G_2(x,y)+\psi_2(x,y)\right\}^{\frac{1}{\beta}} + D_n(x,y)^r\\
&\leqslant (3C)^{-\frac{1}{\beta}}2^{\frac{1}{\beta}}\Big\{C^{\frac{1}{\beta}}[d(x,y)^r-D_n(x,y)^r+G_1(x,y)+\psi_1(x,y)] +[G_2(x,y)+\psi_2(x,y)]^{\frac{1}{\beta}}\Big\} + D_n(x,y)^r\\
&=\left(\tfrac{2}{3}\right)^{\frac{1}{\beta}}\left[d(x,y)^r-D_n(x,y)^r+G_1(x,y)+\psi_1(x,y)\right] +\left(\tfrac{2}{3}\right)^{\frac{1}{\beta}}C^{-\frac{1}{\beta}}\left[G_2(x,y)+\psi_2(x,y)\right]^{\frac{1}{\beta}}+ D_n(x,y)^r\\
&\leqslant d(x,y)^r-D_n(x,y)^r+G_1(x,y)+\psi_1(x,y) +\left(\tfrac23\right)^{\frac{1}{\beta}}C^{-\frac{1}{\beta}}3^{\frac{1}{\beta}}\left[G_2(x,y)^{\frac{1}{\beta}}+\psi_2(x,y)^{\frac{1}{\beta}}\right] + D_n(x,y)^r\\
&=d(x,y)^r+G_1(x,y)+\psi_1(x,y)+\left(\tfrac{C}{2}\right)^{-\frac{1}{\beta}}G_2(x,y)^{\frac{1}{\beta}}+\left(\tfrac{C}{2}\right)^{-\frac{1}{\beta}}\psi_2(x,y)^{\frac{1}{\beta}},
\end{align*}
where we have used the inequality $(a+b)^\frac{1}{\beta}\leqslant 2^{\frac1\beta-1}(a^\frac1\beta+b^\frac1\beta)$ for $a,b\geqslant 0$.

Define $g(\alpha_1, \cdots, \alpha_m)=g_1(\alpha_1, \ldots, \alpha_m)+\left(\frac{C}{2}\right)^{-\frac{1}{\beta}}g_2(\alpha_1, \ldots, \alpha_m)^{\frac{1}{\beta}}$, where $g$ is a function from $(\mathbb{R}^+)^m$ to $\mathbb{R}^+$, and $\psi\colon X \times X \to \mathbb{R}^+$ by $\psi(x,y)=\psi_1(x,y)+\left(\frac{C}{2}\right)^{-\frac{1}{\beta}}\psi_2(x,y)^{\frac{1}{\beta}}$. Note that $g(0,...,0)=0$, $g$ is continuous and nondecreasing with respect to each variable. Furthermore, $\psi\in \operatorname{contr}(B_{t-nT})$ for all $n\in\mathbb{N}$. Since $v$ is the inverse function of $u$ and it is increasing, for $G(x,y)=g(\rho_1(x,y),\ldots,\rho_m(x,y))$, it follows that
\begin{equation}\label{infzero}
D_n(x,y)^r \leqslant v\big(d(x,y)^r+G(x,y)+\psi(x,y)\big).
\end{equation}

For $A\subset B_{t-T}$ and $\varepsilon>0$, there exist sets $E_1, \cdots, E_p$ such that
\begin{align*}
A\subset \bigcup_{j=1}^p E_j \quad \hbox{ and } \quad \operatorname{diam}(E_j)<\kappa(A)+\varepsilon \hbox{ for }j=1,\cdots,p.
\end{align*}
If $\{x_i\}\subset A$, there exist $j\in\{1,\cdots,p\}$ and a subsequence $\{x_{i_k}\}\subset E_j$, and thus
\begin{align}\label{infa}
d(x_{i_k},x_{i_\ell})\leqslant \operatorname{diam}(E_j)<\kappa(A)+\varepsilon \hbox{ for all } k,\ell \in\mathbb{N}.
\end{align}
Since $\rho_1,\cdots, \rho_m$ are precompact on $B_{t-T}$ and $\psi\in \operatorname{contr}(B_{t-T})$ we have
\begin{align}\label{infb}
\liminf_{k,\ell\to \infty} G(x_{i_k},x_{i_\ell})=0 \quad \hbox{ and } \quad \liminf_{k,\ell\to \infty} \psi(x_{i_k},x_{i_\ell})=0.
\end{align}
Joining \eqref{infzero}, \eqref{infa} and \eqref{infb}, we obtain 
\begin{align*}
\liminf_{k,\ell\to \infty}D_1(x_{i_k},x_{i_\ell})^r & \leqslant \liminf_{k,\ell\to \infty}v\big(d(x_{i_k},x_{i_\ell})^r+G(x_{i_k},x_{i_\ell})+\psi(x_{i_k},x_{i_\ell})\big). \\
&\leqslant v\left(\left(\kappa(A)+\varepsilon\right)^r\right).
\end{align*}
Since $\varepsilon>0$ is arbitrary, we conclude that for any sequence $\{x_i\}\subset A$ we have
\begin{align*}
\liminf_{m,p\to \infty}D_1(x_m,x_p)^r \leqslant \liminf_{k,\ell\to \infty}D_1(x_{i_k},x_{i_\ell})^r\leqslant v\left(\kappa(A)^r\right).
\end{align*}

Now, let $A\subset B_{t-2T}$, $\varepsilon>0$ and $\{x_i\}\subset A$. As before, there exists a subsequence $\{x_{i_k}\}$ for which $d(x_{i_k},x_{i_\ell})<\kappa(A)+\varepsilon$ for all $k,\ell\in\mathbb{N}$. Since $\rho_1,\cdots, \rho_m$ are precompact on $B_{t-T}$ and $B_{t-2T}$, $\psi\in \operatorname{contr}(B_{t-T})\cap \operatorname{contr}(B_{t-2T})$ and $S_2B_{t-2T}\subset B_{t-T}$, we obtain
\begin{align*}
& \liminf_{k,\ell\to \infty} G(S_2x_{i_k},S_2x_{i_\ell}) = 0, \quad  \liminf_{k,\ell\to \infty}G(x_{i_k},x_{i_\ell}) = 0,\\
& \liminf_{k,\ell\to \infty} \psi(S_2x_{i_k},S_2x_{i_\ell})=0 \quad \hbox{ and } \quad \liminf_{k,\ell\to \infty}\psi(x_{i_k},x_{i_\ell})=0.
\end{align*}
Since for any $x,y\in B_{t-2T}$,
\begin{align*}
d(S(t,t-2T)& x,S(t,t-2T)y)^r =d(S_1S_2x,S_1S_2y)^r = D_1(S_2x,S_2y)^r\\
&\leqslant v\big(d(S_2x,S_2y)^r+G(S_2x,S_2y)+\psi(S_2x,S_2y)\big)\\
& = v\big(D_2(x,y)^r+G(S_2x,S_2y)+\psi(S_2x,S_2y)\big)\\
&\leqslant v\Big(v\big(d(x,y)^r+G(x,y)+\psi(x,y)\big)+G(S_2x,S_2y)+\psi(S_2x,S_2y)\Big),
\end{align*}
we obtain
\begin{align*}
\liminf_{m,p\to \infty}& d(S(t,t-2T)x_m,S(t,t-2T)x_p)^r\\
&\leqslant \liminf_{k,\ell\to \infty}d(S(t,t-2T)x_{i_k},S(t,t-2T)x_{i_\ell})^r \leqslant v^2\left(\left(\kappa(A)+\varepsilon\right)^r\right).
\end{align*}
Since $\varepsilon>0$ is arbitrary we conclude, for any sequence $\{x_n\}\subset A$, that
\begin{align*}
&\liminf_{m,p\to \infty}d(S(t,t-2T)x_m,S(t,t-2T)x_p)^r\leqslant v^2\left(\kappa(A)^r\right).
\end{align*}

Inductively, for any $n\in \mathbb{N}$, $A\subset B_{t-nT}$ and $\{x_i\}\subset A$ we obtain
\begin{align}\label{xcva}
\liminf_{m,p\to \infty}d(S(t,t-nT)x_m,S(t,t-nT)x_p)^r\leqslant v^n\left(\kappa(A)^r\right).
\end{align}

We claim that for $n\in \mathbb{N}$, $n\geqslant \frac{T_1}{T}$ and $A\subset B_{t-nT}$ we have
\begin{align}\label{xcvb}
\kappa\left(S(t,t-nT)A\right)^r\leqslant 2^r v^n\left(\kappa(A)^r\right)\leqslant 2^rv^n(M^r).
\end{align} 
Assume that the first inequality in \eqref{xcvb} fails. Then we can choose $a>0$ such that 
$
2^r v^n(\kappa(A)^r) < a< \kappa(S(t,t-nT)A)^r.
$
This implies that
\begin{align*}
\beta(S(t,t-nT)A)\geqslant \frac12 \kappa(S(t,t-nT)A)>\frac{a^{1/r}}{2},
\end{align*}
that is, $S(t,t-nT)A$ has no finite cover of balls of radius less than or equal to $\frac{a^{1/r}}{2}$. Take an arbitrary $x_1\in A$. Then, there exists $x_2\in A$ such that $d(S(t,t-nT)x_1, S(t,t-nT)x_2)>\frac{a^{1/r}}{2}$, for otherwise  $S(t,t-nT)A\subset \overline{B}_{\frac{a^{1/r}}{2}}(S(t,t-nT)x_1)$. Following this idea, there exists $x_3\in A$ such that $d(S(t,t-nT)x_3,S(t,t-nT)x_i)>\frac{a^{1/r}}{2}$  for $i=1,2$, otherwise $S(t,t-nT)A$ would be contained in the union of two balls of radius $\frac{a^{1/r}}{2}$. This process gives us a sequence $\left\{x_i\right\}_{i\in\mathbb{N}}\subset A$ such that $d(S(t,t-nT)x_i,S(t,t-nT)x_j)>\frac{a^{1/r}}{2}$ for all $i\neq j$. Therefore
\begin{align*}
d(S(t,t-nT)x_i,S(t,t-nT)x_j)^r>\frac{a}{2^r}>v^n(\kappa(A)^r),
\end{align*}
which contradicts \eqref{xcva}. The second inequality of \eqref{xcvb} follows immediately from the fact that $v^n$ is non-decreasing.

From Proposition \ref{functionsuv}, there exists $n_0\in \mathbb{N}$, $n_0\geqslant \frac{T_1}{T}$, such that for $n\geqslant n_0$ we have
\[
v^n(M^r) \leqslant \left[(n-n_0)\left(\tfrac1\beta-1\right)(1+3C)^{-\frac1\beta}+M^\frac{r(\beta-1)}{\beta}\right]^{\frac{\beta}{\beta-1}}.
\]
Hence, if $n\geqslant n_0$ and $A\subset B_{t-nT}$ we have
\begin{align*}
\kappa(S(t,t-nT)A)&\leqslant 2v^n\left(M^r\right)^{1/r}\leqslant 2\left[(n-n_0)\left(\tfrac{1}{\beta}-1\right)(1+3C)^{-\frac{1}{\beta}}+M^{\frac{r(\beta-1)}{\beta}}\right]^{\frac{\beta}{r(\beta-1)}}.
\end{align*} 
In particular, for $n\geqslant n_0$ we obtain
\begin{align*}
\kappa(S(t,t-nT)B_{t-nT})\leqslant 2\left[(n-n_0)\left(\tfrac{1}{\beta}-1\right)(1+3C)^{-\frac{1}{\beta}}+M^{\frac{r(\beta-1)}{\beta}}\right]^{\frac{\beta}{r(\beta-1)}}.
\end{align*} 

If $s\geqslant (n_0+1)T$ and $n\in \mathbb{N}$ is such that $\frac{s}{T}-1<n\leqslant \frac{s}{T}$, we have $n>n_0$ and since $S(t-nT,t-s)B_{t-s}\subset B_{t-nT}$, 
\begin{align*}
\kappa(S(t,t-s)B_{t-s})&=\kappa\left(S\left(t,t-nT\right)S\left(t-nT,t-s\right)B_{t-s}\right)\leqslant \kappa\left(S\left(t,t-nT\right)B_{t-nT}\right)\\
&\leqslant 2\left[\left(n-n_0\right)\left(\tfrac{1}{\beta}-1\right)(1+3C)^{-\frac{1}{\beta}}+M^{\frac{r(\beta-1)}{\beta}}\right]^{\frac{\beta}{r(\beta-1)}}\\
&\stackrel{(\ast)}{\leqslant} 2\left[\left(\tfrac{s}{T}-1-n_0\right)\left(\tfrac{1}{\beta}-1\right)(1+3C)^{-\frac{1}{\beta}}\right]^{\frac{\beta}{r(\beta-1)}}\\
&=2\Big[\tfrac{1}{T}\left(\tfrac{1}{\beta}-1\right)(1+3C)^{-\frac{1}{\beta}}s-(1+n_0)\left(\tfrac{1}{\beta}-1\right)(1+3C)^{-\frac{1}{\beta}}\Big]^{\frac{\beta}{r(\beta-1)}}\\
& =2(\omega s-\eta)^{\frac{\beta}{r(\beta-1)}},
\end{align*}
where 
\[
\omega=\tfrac{1}{T}\left(\tfrac{1}{\beta}-1\right)(1+3C)^{-\frac{1}{\beta}} \hbox{ and } \eta=(1+n_0)\left(\tfrac{1}{\beta}-1\right)(1+3C)^{-\frac{1}{\beta}},
\]
and in $(\ast)$ we used the fact that the exponent $\frac{\beta}{r(\beta-1)}$ is negative. Now, there exist $C_1>0$ and $s_0>0$ such that $2(\omega s-\eta)^{\frac{\beta}{r(\beta-1)}}\leqslant C_1(\omega s)^{\frac{\beta}{r(\beta-1)}}$ for $s\geqslant s_0$.
Then, taking $s\geqslant s_1:=\max\left\{s_0, (n_0+1)T\right\}$, we have
\[
\kappa(S(t,t-s)B_{t-s})\leqslant C_1(\omega s)^{\frac{\beta}{r(\beta-1)}}
\]
for $s\geqslant s_1$.
\end{proof}

Now we are able to prove Theorem \ref{corMain}.

\begin{proof}[Proof of Theorem \ref{corMain}]
It follows from Proposition \ref{thmplm} that $S$ is $\varphi$-pullback $\kappa$-dis\-sip\-a\-tive, with the decay function $\varphi$ given by $\varphi(s)=s^{\frac{\beta}{r(\beta-1)}}$, and therefore, from Theorem \ref{continuouscase} there exists a generalized $\varphi$-pullback attractor $\hat{M}$ for $S$ with $\hat{M}\subset \hat{B}$. Since $\hat{B}$ is uniformly bounded, so is $\hat{M}$. 
\end{proof}

\subsection{Pullback attractors}

In this section we want to relate the notion of generalized $\varphi$-pullback attractor and the one of pullback attractor, just as in the exponential case. For a different and more detailed approach on this subject we refer to \cite{CRLBook}.  More specifically, we state and prove simple the following result:

\begin{theorem}\label{theo:GenimpliesPA}
Let $S$ be a $\varphi$-pullback $\kappa$-dissipative in $X$ with a backwards bounded generalized $\varphi$-pullback attractor $\hat{M}$. Then $S$ has a pullback attractor $\hat{A}$, with $\hat{A}\subset \hat{M}$.
\end{theorem}

The next proposition is a consequence of a well known result in the theory of pullback attractors, and can be easily proven using, for instance, the results of \cite{CRLBook}, just noticing that a $\varphi$-pullback $\kappa$-dissipative evolution process is also \textit{asymptotically compact}. 

\begin{proposition}\label{teoprincipal}
Let $S$ a $\varphi$-pullback $\kappa$-dissipative evolution process on $X$ with $\hat{B}\in\mathfrak{F}$ a backwards bounded and pullback absorbing family. Then $S$ has a pullback attractor.
\end{proposition}

With that, we are able to present the proof of Theorem \ref{theo:GenimpliesPA}.

\begin{proof}[Proof of Theorem \ref{theo:GenimpliesPA}]
    Since $\hat{M}$ is a backwards bounded generalized $\varphi$-pullback attractor, given $r>0$ we can see that the family $\hat{B}$ defined by $B_t = \mathcal{O}_r(M_t)$, is a backwards bounded absorbing family for $S$. Hence, from Proposition \ref{teoprincipal} it follows that $S$ has a pullback attractor $\hat{A}$. Lastly, since $\hat{M}$ is a closed pullback attracting family, the minimality property of $\hat{A}$ shows that $\hat{A}\subset \hat{M}$, and the proof is complete.
\end{proof}

\section{Application to a nonautonomous wave equation} \label{Application}

Inspired by the works \cite{ZhaoZhao2020}, \cite{ZhaoZhong2022}, and \cite{yan2023long}, and following their main ideas, we study the nonautonomous wave equation with non-local weak damping and anti-damping \eqref{ourproblem}, and prove the existence of a \textit{generalized polynomial pullback attractor}. We recall that \eqref{ourproblem} is given by
\begin{equation*}
\left\{\begin{aligned}
& u_{tt}(t,x)-\Delta u(t,x)+k(t)\|u_t(t, \cdot)\|^p_{L^2(\Omega)}u_t(t,x)+f(t,u(t,x))\\
& \hspace{130pt} = \int_{\Omega}K(x,y)u_t(t,y)dy+h(x), (t,x)\in [s,\infty)\times \Omega,\\
& u(t,x) = 0, (t,x)\in \left[s,\infty \right) \times \partial\Omega,\\
& u(s,x)=u_0(x), \ u_t(s,x)=u_1(x), x\in \Omega.
\end{aligned} \right.
\end{equation*}
where $\Omega\subset \mathbb{R}^3$ is a bounded domain with smooth boundary $\partial\Omega$. Here, $u$ represents the displacement of a wave in $\Omega$, subjected to a non-local damping $k(t)\|u_t(t,\cdot)\|^p_{L^2(\Omega)}$ and an anti-damping $\int_\Omega K(x,y)u_t(t,y)dt$. We assume the following:

\begin{enumerate}[leftmargin=*,label={(\bfseries H$_{\arabic*})$}]
\item \label{cond1} $K\in L^2(\Omega \times \Omega)$ and we set $K_0:=\|K\|_{L^2(\Omega\times \Omega)}$,
\item $k\colon \mathbb{R}\rightarrow (0,\infty)$ is a continuous function with $0<k_0\leqslant k(t)\leqslant k_1$ for all $t\in \mathbb{R}$, where $k_0, \ k_1$ are positive constants, 
\item $p>0$,
\item $h\in L^2(\Omega)$ and we set $h_0:=\|h\|_{L^2(\Omega)}$,
\item \label{condf} $f\in C^1(\mathbb{R}^2, \mathbb{R})$ satisfies
\begin{equation*}
\liminf_{|v|\to \infty}\Big(\inf_{t\in\mathbb{R}}\frac{\partial f}{\partial v}(t,v)\Big) > -\lambda_1, \quad \hbox{ and } \quad  \liminf_{|v|\to \infty}\Big(\inf_{t\in\mathbb{R}}\frac{f(t,v)}{v}\Big) > -\lambda_1,
\end{equation*}
and there exists a constant $c_0>0$ such that for all $t,v\in \mathbb{R}$ we have
\begin{equation*}
|f(t,0)| \leqslant c_0,\quad  \left|\frac{\partial f}{\partial v}(t,v)\right| \leqslant c_0 (1+|v|^2), \quad \hbox{ and } 
\quad \left|\frac{\partial f}{\partial t}(t,v)\right| \leqslant c_0,
\end{equation*}
where  $\lambda_1>0$ is the first eigenvalue of the negative Laplacian operator $-\Delta$ with Dirichlet boundary conditions in $\Omega$, that is, of the operator $A:=-\Delta \colon H^1_0(\Omega)\cap H^2(\Omega)\subset L^2(\Omega)\to L^2(\Omega)$, which is positive and selfadjoint, with compact resolvent.
\item \label{cond6} For $t,v\in \mathbb{R}$ we define the function $F(t,v)=\displaystyle\int_{0}^v f(t,\xi)d\xi$, for which we assume that $c_0>0$ previously defined is such that for all $t\in \mathbb{R}$ we have
\[
\int_{\mathbb{R}}\left|\frac{\partial F}{\partial t}(t,v)\right|dv \leqslant c_0
\]
\end{enumerate}

The main result of this section is the following:

\begin{theorem}\label{App:PolAtt}
Consider the decay function $\varphi(t)=t^{-\frac1p}$ and assume that \ref{cond1}-\ref{cond6} hold true. Then the evolution process $S=\{S(t,s)\colon t\geqslant s\}$ in $X:=H^1_0(\Omega)\times L^2(\Omega)$ associated with \eqref{ourproblem} possesses a bounded generalized $\varphi$-pullback attractor $\hat{M}$ in $X$. Furthermore, $S$ has a pullback attractor $\hat{A}$, with $\hat{A}\subset \hat{M}$.
\end{theorem}

Our efforts from now on are dedicated to present the proof of Theorem \ref{App:PolAtt}.

\subsection{Auxiliary estimates}

Now we present a few estimates, regarding the functions $f$ and $F$, that will help us in the sections to come. In what follows, unless stated otherwise, $C$ denotes an independent positive constant for which its value may vary from one result to another, from one line to another, and even in the same line. Also, from now on, we are assuming that conditions \ref{cond1}-\ref{cond6} hold true. To simplify the notation, we will omit the $(\Omega)$ from the subscript of the norms. For instance, we write $\|\cdot\|_{L^2}$ instead of $\|\cdot\|_{L^2(\Omega)}$, $\|\cdot\|_{H^1_0}$ instead of $\|\cdot\|_{H^1_0(\Omega)}$, and henceforth. Also, for a Banach space $Z$ we will denote $\overline{B}^Z_R$ the closed ball with radius $R$ centered at $0$ in $Z$.

The estimates presented here are very useful to us, but we omit their proofs, since they are straightforward.

\begin{proposition}\label{funcFimp}
For all $t,v,w\in \mathbb{R}$ we have:
\begin{enumerate}[label={(\roman*)}]
    \item \label{funcFimp.a} $|f(t,v)|\leqslant 2c_0(1+|v|^3)$;
    \item \label{funcFimp.a2} $|f(t,v)-f(t,w)|\leqslant 2c_0(1+|v|^2+|w|^2)|v-w|$;
    \item \label{funcFimp.b} $|F(t,v)|\leqslant 4c_0(1+|v|^4)$;
    \item \label{funcFimp.b2} $|F(t,v)-F(t,w)|\leqslant 4c_0(1+|v|^3+|w|^3)|v-w|$;
    \item \label{funcFimp.c} $\left|\frac{\partial F}{\partial t}(t,v)-\frac{\partial F}{\partial t}(t,w)\right| \leqslant 2c_0|v-w|$.
\end{enumerate}
\end{proposition}

In what follows, we use $\hookrightarrow$ to denote the continuous inclusions, and $\hookrightarrow\hookrightarrow$ to denote compact inclusions. Using the continuous inclusion $H^1_0(\Omega)\hookrightarrow L^6(\Omega)$, and Poincar\'e's inequality, we can prove the next result.

\begin{lemma}\label{wvtlema}
There exists a constant $L_0>0$ such that 
\[
\|f(t,v)-f(t,w)\|_{L^2}\leqslant L_0(1+\|v\|^2_{H^1_0}+\|w\|^2_{H^1_0}) \|v-w\|_{H^1_0} \quad \hbox{ for all } v,w\in H^1_0(\Omega).
\]
In particular
\[
\|f(t,v)-f(t,w)\|_{L^2} \leqslant L_0(1+2R^2)\|v-w\|_{H^1_0} \quad \hbox{ for all } v,w\in \overline{B}^{H^1_0(\Omega)}_R.
\]
\end{lemma}

From the definition of $\liminf$, the proof of the following proposition is trivial.

\begin{proposition}\label{propax}
Given $\mu_0<\lambda_1$ there exists $M=M(\mu_0)>0$ such that 
\[
\inf_{t\in \mathbb{R}}\frac{\partial f}{\partial v}(t,v)>-\mu_0 \quad \hbox{ and } \quad \inf_{t\in \mathbb{R}}\frac{f(t,v)}{v}>-\mu_0 \quad \hbox{ for all } |v|>M.
\]
\end{proposition}

Furthermore, we have

\begin{proposition}\label{propxnh}
For each $M>0$ we have $\int_{-M}^M|f(t,v)|dv \leqslant 8c_0(1+M^4)$ for all $t\in \mathbb{R}$.
\end{proposition}

\begin{proposition}\label{propF}
Fixing $0<\mu_0<\lambda_1$ and considering $M=M(\mu_0)>0$ given by Proposition \ref{propax}, we have
\[
F(t,v)\geqslant -\tfrac{\mu_0+\lambda_1}{4}v^2-8c_0(1+M^4) \quad \hbox{ for } |v|>M \hbox{ and } t\in \mathbb{R},
\]
and 
\[
|F(t,v)|\leqslant 8c_0(1+M^4) \quad \hbox{ for } |v|\leqslant M \hbox{ and } t\in \mathbb{R}.
\]
\end{proposition}

\begin{proposition}\label{proppp}
Fixing $0<\mu_0<\lambda_1$ and considering $M=M(\mu_0)>0$ given by Proposition \ref{propax}, there exists a constant $e_0>0$ such that
\[
F(t,v)\leqslant vf(t,v)+\frac{\mu_0}{2}v^2+e_0 \quad \hbox{ for } |v|>M \hbox{ and } t\in\mathbb{R}.
\]
\end{proposition}

\medskip \noindent \textbf{Translation of the problem} \medskip

In order to use our knowledge of the autonomous problem, presented in \cite{ZhaoZhao2020,ZhaoZhong2022}, we use the translations of the original nonautonomous problem, in order to obtain a problem defined in $[0,\infty)$ rather than on $[s,\infty)$. To be more specific, fixing $s\in \mathbb{R}$ and setting $v(t,x):=u(t+s,x)$ for $t\geqslant 0$ and $x\in\Omega$, we formally have
\[
\begin{aligned}
&v_{tt}(t,x)-\Delta v(t,x)+k_s(t)\left\|v_t(t, \cdot)\right\|_{L^2}^pv_t(t,x)+f_s(t, v(t,x)) -\int_{\Omega}K(x,y)v_t(t,y)dy-h( x)\\
&=u_{tt}(t+s, x)-\Delta u(t+s, x)+k_s(t)\left\|u_t(t+s, \cdot)\right\|_{L^2}^pu_t(t+s, x)\\
&\qquad+f_s(t,u(t+s,x))-\int_{\Omega}K(x,y)u_t(t+s,y)dy-h(x)=0,
\end{aligned}
\]
where the boundary and initial conditions become
\begin{align*}
& v(t,x) = u(t+s,x) = 0 \quad \hbox{ for } (t,x)\in \left[0,\infty \right) \times \partial\Omega,\\
& v(0,x)=u(s,x)=u_0(x), \ v_t(0,x)=u_t(s,x)=u_1(x) \quad \hbox { for } x\in \Omega.
\end{align*}
Thus, we will study the boundary and initial conditions problem
\begin{equation}\label{eqondap}\tag{tNWE}
\left\{\begin{aligned}
&v_{tt}(t,x)-\Delta v(t,x)+k_s(t)\left\|v_t(t, \cdot)\right\|_{L^2}^pv_t(t,x)+f_s(t,v(t,x))\\
& \hspace{130pt} = \int_{\Omega}K(x,y)v_t(t,y)dy + h(x), \ (t,x)\in [0,\infty)\times \Omega,\\
& v(t,x) = 0, \ (t,x)\in [0,\infty)\times \partial \Omega,\\
& v(0,x)=u_0(x), \ v_t(0,x) = u_1(x), \ x\in \Omega,
\end{aligned}\right.
\end{equation}
instead of \eqref{ourproblem}. This problem is equivalent to the initial one, but with the nonautonomous terms being $k_s(\cdot)=k(\cdot+s)$ and $f_s(\cdot,\cdot)=f(\cdot+s,\cdot)$ instead of $k$ and $f$. Additionally, we denote $F_s(\cdot , \cdot)=F(\cdot + s, \cdot)$.

\subsection{Well-posedness}

We will use the classical Semigroups Theory to obtain the existence of local weak solutions, continuous dependence on initial data and a result regarding the continuation of solution. To that end, we transform \eqref{eqondap} into an abstract Cauchy problem in an appropriate phase space.

\newcommand{\vetor}[2]{\left[\begin{smallmatrix} #1 \\ #2 \end{smallmatrix}\right]}

Taking $w=v_t$ in \eqref{eqondap}, we obtain 
\[
w_t=v_{tt}=\Delta v-k_s(t)\|v_t\|^p_{L^2}v_t-f_s(t,v)+\int_{\Omega}K(x,y)v_t(t,y)dy+h(x),
\]
and thus
\begin{equation*}
\begin{aligned}
\frac{d}{dt}\begin{bmatrix}
    v\\w
\end{bmatrix}&=\begin{bmatrix}
    v_t\\w_t
\end{bmatrix}=
\begin{bmatrix} w \\ \Delta v-k_s(t)\|v_t(t,\cdot)\|^p_{L^2}v_t -f_s(t,v)+\int_{\Omega}K(x,y)v_t(t,y)dy+h(x)\end{bmatrix}\\
&=\begin{bmatrix} 0 & I \\ \Delta & 0 \end{bmatrix}\begin{bmatrix}
    v\\w
\end{bmatrix}+\begin{bmatrix} 0 \\ \int_{\Omega}K(x,y)w(t,y)dy+h(x)-f_s(t,v)-k_s(t)\left\|w\right\|^p_{L^2}w \end{bmatrix}.
\end{aligned}
\end{equation*}
Setting $X:=H^1_0(\Omega)\times L^2(\Omega)$, $V=\vetor{v}{w}$, $V_0=\vetor{u_0}{u_1}$, $\mathcal{A}=\left[\begin{smallmatrix} 0 & I \\ \Delta & 0 \end{smallmatrix}\right]$ and $\mathcal{G}(t,V)=\vetor{0}{G(t,V)}$, where 
\[
G(t,V)=\int_{\Omega}K(x,y)w(t,y)dy+h(x)-f_s(t,v)-k_s(t)\left\|w\right\|^p_{L^2(\Omega)}w,
\]
we can represent \eqref{eqondap} by an abstract Cauchy problem
\begin{equation}\label{Abs.Ev.Equation}\tag{ACP}
\left\{\begin{aligned}
    &\frac{dV}{dt}=\mathcal{A}V+\mathcal{G}(t,V), \ t>0\\
    &V(0)=V_0\in X.
\end{aligned} \right.
\end{equation}
Clearly $X$ is a Hilbert space with the inner product defined by          
\begin{equation*}
\left\langle \vetor{v_1}{w_1},\vetor{v_2}{w_2} \right \rangle_X=\langle v_1,v_2\rangle_{H^1_0(\Omega)}+\langle w_1, w_2 \rangle_{L^2(\Omega)},
\end{equation*}
with associate norm
\begin{equation*}
\left\|\vetor{v}{w}\right\|_X^2= \|v\|^2_{H^1_0} + \| w\|^2_{L^2}.
\end{equation*}
The operator $\mathcal{A}\colon D(\mathcal{A})\subset X\to X$ has its usual domain $D(\mathcal{A})=[H^1_0(\Omega)\cap H^2(\Omega)]\times H^1_0(\Omega)$. This operator is closed, densely defined and $m$-dissipative. From the Lumer-Philips Theorem, $\mathcal{A}$ generates a $C^0-$semigroup of contractions $\{e^{\mathcal{A}t}\colon t\geqslant 0\}\subset \mathcal{L}(X)$.

To deal with problem \eqref{Abs.Ev.Equation}, we look into the nonlinear term $\mathcal{G}(t,V)$.

\begin{proposition}
Given $R>0$ there exists $L_R \geqslant 0$ such that
\[
\|\mathcal{G}(t,V_1)-\mathcal{G}(t,V_2)\|_X\leqslant L_R\|V_1-V_2\|_X \quad \hbox{ for all } V_1,V_2\in \overline{B}^X_R \hbox{ and } t\in \mathbb{R},
\]
that is $\mathcal{G}\colon \mathbb{R}\times X\to X$ is locally Lipschitz in the second variable, uniformly for $t\in \mathbb{R}$. Furthermore, $\mathcal{G}$ takes bounded subsets of $\mathbb{R}\times X$ into bounded subsets of $X$.
\end{proposition}
\begin{proof}
Let $V_1=\vetor{v_1}{w_1}$ and $V_2=\vetor{v_2}{w_2}$ taken in $\overline{B}_R^X$. We can assume, without loss of generality, that $\|w_1\|_{L^2}\geqslant \|w_2\|_{L^2}$. Note that
\[
\begin{aligned}
&\|G(t,V_1)-G(t,V_2)\|_{L^2}\leqslant \left\|\int_{\Omega}K(x,y)(w_1-w_2)(t,y) dy\right\|_{L^2} + \|f_s(t,v_1)-f_s(t,v_2)\|_{L^2}\\
&\qquad+k_s(t)\left\|\|w_1\|^p_{L^2}w_1-\|w_2\|^p_{L^2}w_2\right\|_{L^2}.
    \end{aligned}
\]
The first term on the right hand side is estimated by:
\[
\left\|\int_{\Omega}K(x,y)(w_1-w_2)(t,y) dy\right\|_{L^2}\leqslant K_0\|w_1-w_2\|_{L^2}.
\]
Applying Lemma \ref{wvtlema}, we estimate the second term:
\[
\|f_s(t,v_1)-f_s(t,v_2)\|_{L^2}\leqslant L_0(1+2R^{2}) \|v_1-v_2\|_{H^1_0},
\]
since $v_1,v_2\in \overline{B}^{H^1_0(\Omega)}_R$.  For the third one, it follows from Proposition \ref{upLipschitz} that 
\[
k_s(t) \big\|\|w_1\|^p_{L^2}w_1-\|w_2\|^p_{L^2}w_2\big\|_{L^2}
\leqslant  k_1R^p(1+p)\|w_1-w_2\|_{L^2}.
\]
Thus, we obtain
\begin{equation*}
\begin{aligned}
\|G(t,V_1)-G(t,V_2)\|_{L^2} & \leqslant L_0(1+2R^{2})\|v_1-v_2\|_{H^1_0} +(K_0 + k_1R^p(1+p))\|w_1-w_2\|_{L^2}.
\end{aligned}
\end{equation*}
Setting
\[
L_R:= (2\max\{L_0^2(1+2R^{2})^2,\ (K_0+k_1R^p(1+p))^2\})^\frac12,
\]
we obtain
\[
\|G(t,V_1)-G(t,V_2)\|_{L^2}^2 \leqslant L_R^2 (\|v_1-v_2\|_{H^1_0}^2 + \|w_1-w_2\|_{L^2}^2) = L_R^2\|V_1-V_2\|_X^2,
\]
and the first claim follows easily, since $\|\mathcal{G}(t,V_1)-\mathcal{G}(t,V_2)\|_X = \|G(t,V_1)-G(t,V_2)\|_{L^2}$. The second claim is also simple to prove, given that $\|\mathcal{G}(t,0)\|_X$ is bounded, uniformly for $t\in \mathbb{R}$.
\end{proof}

Thanks to all these properties and the results of \cite{Pazy}, we deduce the following:

\begin{proposition}\label{solutions.Ex}
For any given $V_0\in X$ there exists a unique maximal weak solution $V(\cdot,V_0)\colon [0,\tau_{max})$ $\to X$ of \eqref{Abs.Ev.Equation}, that is, a continuous function such that
\[
V(t,V_0) = e^{\mathcal{A}t}V_0 + \int_0^t e^{\mathcal{A}(t-\tau)} \mathcal{G}(\tau,V(\tau,V_0))d\tau \quad \hbox{ for } t\in [0,\tau_{max}),
\]
such that either $\tau_{max}=\infty$ or $\tau_{max}<\infty$ and 
$
\limsup\limits_{t\to \tau_{max}^-}\|V(t,V_0)\|_X = \infty.
$
\end{proposition}

Writing $V(t,V_0)=\vetor{v(t)}{w(t)}$ we can see that
$
v \in C([0,\tau_{max}),H^1_0(\Omega)) \hbox{ and } w\in C([0,\tau_{max}),L^2(\Omega)).
$
Since $w=v_t$, we conclude that
$
v\in C([0,\tau_{max}),H^1_0(\Omega))\cap C^1([0,\tau_{max}),L^2(\Omega))
$
is the unique maximal weak solution to \eqref{eqondap}. This solution satisfies 
\begin{align}
\frac{d}{dt} \int_\Omega v_t(t)\psi dx &+ \int_\Omega \nabla v(t) \nabla \psi dx + k_s(t)\|v_t(t)\|^p_{L^2}\int_\Omega v_t(t)\psi dx + \int_\Omega f_s(t,v(t))\psi dx  \nonumber \\
& = \int_{\Omega \times \Omega} K(x,y)v_t(t,y)\psi(x)dydx + \int_\Omega h\psi dx, \nonumber
\end{align}
for all $\psi\in H^1_0(\Omega)$. Also, we have the map $v_{tt}\colon (0,\tau_{\max})\to H^{-1}(\Omega)$ well-defined and it satisfies
\begin{equation}\label{eq:normaHm1}
\begin{aligned}
\|v_{tt}(t)\|_{H^{-1}} & \leqslant \|v(t)\|_{H^1_0} + k_s(t) \lambda_1^{-\frac12}\|v_t(t)\|^{p+1}_{L^2} + \lambda_1^{-\frac12}\|f_s(t,v(t))\|_{L^2}\\
& \qquad \qquad +K_0\lambda_1^{-\frac12}\|v_t(t)\|_{L^2} + \lambda_1^{-\frac12}h_0.
\end{aligned}
\end{equation}

\bigskip \noindent \textbf{Non-explosion in finite time} \medskip

In order to show that $\tau_{max}=\infty$ for all $V_0\in X$, we need to show that the second condition cannot occur. This will be a consequence of some results we will present later (see Remark \ref{remark.Global}) and we know that for each given $V_0\in X$ the solution of \eqref{Abs.Ev.Equation} is defined for all $t\geqslant 0$.  

Fix $s\in \mathbb{R}$ and $V_0:=(u_0,u_1)\in X$. If $V(\cdot,V_0)\colon [0,\infty)\to X$ is the solution of \eqref{Abs.Ev.Equation}, we define for $t\geqslant s$
\begin{equation*}
S(t,s)(u_0,u_1) = V(t-s, V_0).
\end{equation*}
Note that if, for $t\geqslant s$, we set $(u(t),y(t)) = S(t,s)(u_0,u_1)$, then $u(t)=v(t-s)$ and $y(t)=v_t(t-s)=u_t(t)$. Thus $(u(t),u_t(t))$ is the weak solution of \eqref{ourproblem}, and $\{S(t,s)\colon t\geqslant s\}$ defines an evolution process in $X$ associated with \eqref{ourproblem}, provided we have continuity with respect with initial data (which will be proved next).

\bigskip \noindent \textbf{Continuous dependence on the initial data: a Lipschitz condition} \medskip

Let $V_1,V_2\in X$ be the initial data for \eqref{Abs.Ev.Equation}, taken in the closed ball $\overline{B}_{R}^{X}$, for a given $R>0$. Consider $v, w$ the corresponding solutions for \eqref{eqondap} related respectively to these initial data. Later, see Proposition \ref{EsC}, we will show that there exists a constant $c_R\geqslant 0$ such that for all $t\geqslant 0$ we have
\begin{equation}\label{uhbijnx}
\|V(t,V_1)\|_X\leqslant c_R \quad  \hbox{ and } \quad \|V(t,V_2)\|_X\leqslant c_R.
\end{equation}
If $V(t,V_1)=(v(t),v_t(t))$ and $V(t,V_2)=(w(t),w_t(t))$, setting $z:=v-w$, formally we obtain 
\begin{equation}\label{fghjmmm} 
z_{tt}-\Delta z+k_s(t)(\|v_t\|^p_{L^2}v_t-\|w_t\|^p_{L^2}w_t) + f_s(t,v)-f_s(t,w)=\int_{\Omega}K(x,y)z_t(t,y)dy.
\end{equation}

Direct estimates and Proposition \ref{prop.Cp} give us the following result.

\begin{proposition}\label{dcdi}
Given $R>0$, there exists a constant $\gamma_0=\gamma_0(R) \geqslant 0$ such that for $t\geqslant 0$ and $V_1,V_2\in \overline{B}^X_R$, we have
\[
\|V(t,V_1)-V(t,V_2)\|_X \leqslant \|V_1-V_2\|_Xe^{\gamma_0 t}.
\]
\end{proposition}

\subsection{Properties of uniformly pullback absorption}\label{absorbingfamily}

In this section we prove the existence of a closed, positively invariant and bounded pullback absorbing family. This fact, together with the $\varphi$-pullback $\kappa$-dissipativity of the evolution process associated with \eqref{ourproblem} is essential to show the existence of the pullback attractor and of a generalized polynomial pullback attractor.

We fix $0<\mu_0<\lambda_1$ and consider $M=M(\mu_0)>0$ given by Proposition \ref{propax}. For a function $v\in L^2(\Omega)$ we set
\begin{equation}\label{Omega12}
\Omega_1=\{x\in\Omega\colon |v(x)|>M\} \quad \hbox{ and } \quad \Omega_2=\{x\in\Omega\colon  |v(x)|\leqslant M\}.
\end{equation}
From Proposition \ref{propF}, for each fixed $t\in \mathbb{R}$ we obtain
\[
\begin{aligned}
\int_{\Omega}& F(t,v)dx = \int_{\Omega_1}F(t,v)dx + \int_{\Omega_2}F(t,v)dx\\
&\geqslant -\int_{\Omega_1} \left[\tfrac{\mu_0+\lambda_1}{4} |v|^2+ 8c_0(1+M^{4})\right]dx - \int_{\Omega_2} 8c_0(1+M^{4}) dx \geqslant -\tfrac{\mu_0+\lambda_1}{4}\|v\|^2_{L^2} - C_0, 
\end{aligned}
\]
where $C_0:=8c_0(1+M^{4})|\Omega|>0$. That is, we obtained
\begin{equation}\label{observaa}
\int_{\Omega} F(t,v)dx + \tfrac{\mu_0+\lambda_1}{4}\|v\|^2_{L^2}+C_0\geqslant 0.
\end{equation}

For each initial data $V_0=(u_0,u_1)\in X$, define the function $E_s(\cdot,V_0)\colon \mathbb{R}_+\to \mathbb{R}$ given by
\begin{equation}\label{energyy}
\begin{aligned}
 E_s(t,V_0)&=\tfrac{1}{2}(\|v_t\|_{L^2}^2+\|v\|_{H^1_0}^2) + \int_{\Omega}F_s(t,v)dx+\tfrac{\lambda_1+\mu_0}{4}\|v\|^2_{L^2}+C_0,
\end{aligned}
\end{equation}
where $C_0$ is the constant obtained in \eqref{observaa}, and $V(t,V_0)=(v(t),v_t(t))$ for $t\geqslant 0$. For simplicity, we write $E_s(t)$ instead of $E_s(t,V_0)$, but we must keep in mind that this function depends on the initial data $V_0\in X$, since it depends on the solution $V(t,V_0)$.

\bigskip \noindent \textbf{Existence of a uniformly pullback absorbing family}\label{absorbinfamily}\medskip

In this section, we will show that there exists a constant $r_0>0$ such that the family $\hat{B}=\{B_t\}_{t\in \mathbb{R}}$ given by
\begin{equation}\label{def.Babs}
B_t = \overline{B}^X_{r_0} \quad \hbox{ for each } t\in \mathbb{R}
\end{equation}
is a uniformly pullback absorbing family for $S$. More precisely, we will prove the following theorem:

\begin{theorem}[Existence of a pullback absorbing family]\label{teo.existence.paf}
There exists $r_0>0$ such that for each $R>0$ there exists $\tau_0=\tau_0(R)\geqslant 0$ with
\[
\|S(t,t-\tau)(u_0,u_1)\|_X \leqslant r_0,
\]
for all $\tau\geqslant \tau_0$, $t\in \mathbb{R}$ and $(u_0,u_1)\in X$ with $\|(u_0,u_1)\|_X\leqslant R$.
\end{theorem}

Note that for $t\geqslant s$ and $V_0=(u_0,u_1)\in X$, from \eqref{energyy} and \eqref{observaa} we have
\begin{align*}
\|S(t,s)(u_0,u_1)\|_X^2&=\|V(t-s,V_0)\|_X^2 \leqslant 2E_s(t-s).
\end{align*}
Hence, the study of the function $E_s$ is paramount, since this function bounds the norm of $S(t,s)(u_0,u_1)$ for $t\geqslant s$. The proof of Theorem \ref{teo.existence.paf} is not trivial, and to facilitate its comprehension, we present a scheme that will help us with this goal. \medskip

\noindent \textbf{Scheme.}

\par \noindent (I) Define an auxiliary function $V_{\varepsilon}^s$ and prove that it is equivalent to $E_s$, which is done in Proposition \ref{propa401}; \smallskip
\par \noindent (II) Using the function $V_\varepsilon^s$ obtained in (I), prove that the function $E_s$ remains bounded for all times $t\geqslant 0$, uniformly for initial data in bounded subsets of $X$. This result is achieved in Proposition \ref{EsC}; \smallskip
\par \noindent (III) Improve the result obtained in (II), and in Proposition \ref{estimate.taumax} show that there exists a fixed bounded subset of $X$ such that $E_s(t)$ enters this subset if $t$ is sufficiently large, uniformly for initial data in bounded subsets of $X$;
\par \noindent (IV) Using the result from (III) to prove Theorem \ref{teo.existence.paf}.

\medskip \noindent \textsc{Development of (I).} \medskip

For a fixed $\varepsilon>0$ and $(v,v_t)=V(t,V_0)$ a solution to \eqref{Abs.Ev.Equation} we define $V_{\varepsilon}^s(\cdot,V_0)\colon \mathbb{R}_+\to \mathbb{R}$ by
\begin{equation}\label{def.Veps}
V^s_{\varepsilon}(t,V_0)=\tfrac{1}{2}\big(\|v\|^2_{H^1_0} + \|v_t\|^2_{L^2}\big)+\int_{\Omega}F_s(t,v)dx-\int_{\Omega}hvdx+\varepsilon\int_{\Omega}v_tvdx.
\end{equation}
As we did for the function $E_s$, to simplify the notation we simply write $V_\varepsilon^s(t)$ instead of $V_\varepsilon^s(t,V_0)$. However, we must keep in mind that this function depends on the initial data $V_0\in X$, and all the comparisons that are done are using the same initial data $V_0$.

Proceeding as in \cite{ZhaoZhong2022} we obtain the following result.

\begin{proposition}\label{propa401} 
There exist $\varepsilon_0>0$ and a constant $d_0>0$ such that for all $0<\varepsilon \leqslant\varepsilon_0$, $s\in \mathbb{R}$ and $t\geqslant 0$ we have
\[
\tfrac14\left(1-\tfrac{\mu_0}{\lambda_1}\right) E_s(t) - d_0 \leqslant V_\varepsilon^s(t) \leqslant \tfrac54 E_s(t) + d_0.
\]
\end{proposition}

\medskip \noindent \textsc{Development of (II).} \medskip

Following our scheme, we want to prove the following result:

\begin{proposition}\label{EsC}
Given $R>0$ there exists a constant $c_R \geqslant 0$ such that 
\[
E_s(t)\leqslant c_R \quad \hbox{ for all } V_0\in \overline{B}^X_R \hbox{ and } t\geqslant 0.
\]
\end{proposition}

As we did above, we will prove auxiliary lemmas that will come together to prove Proposition \ref{EsC}. However, before that, we will present a remark to justify the existence of global solutions.

\begin{remark}[Solutions to \eqref{Abs.Ev.Equation} are global] \label{remark.Global}
Assume that $V_0\in X$ is given and consider the unique maximal weak solution $V(\cdot,V_0)\colon [0,\tau_{max})\to X$ of \eqref{Abs.Ev.Equation}. The  arguments up until the end of the proof of Proposition \ref{EsC}, remain true for $V(t,V_0)$, or for $E_s(t,V_0)$, when $0\leqslant t<\tau_{max}$. Hence, Proposition \ref{EsC} shows that for some constant $c>0$, we have $E_s(t,V_0)\leqslant c$ for all $0\leqslant t<\tau_{max}$. Therefore, $\|V(t,V_0)\|_X \leqslant c$ for all $0\leqslant t<\tau_{max}$ and shows that if we assume $\tau_{max}<\infty$ we obtain a contradiction with Proposition \ref{solutions.Ex}. Hence $\tau_{max}=\infty$ for all initial data $V_0\in X$.
\end{remark}

Taking the inner product of \eqref{eqondap} by $v_t+\varepsilon v$ in $L^2(\Omega)$ we obtain
\begin{equation*}\label{vepsilont}
\begin{aligned}
\frac{d}{dt}& V^s_{\varepsilon}(t)=-k_s(t)\|v_t\|^{p+2}_{L^2}-\varepsilon k_s(t)\|v_t\|^p_{L^2}\int_{\Omega}v_tv dx +\int_{\Omega}\frac{\partial F_s}{\partial t}(t, v) dx \\ 
&\qquad +\int_{\Omega \times \Omega}K(x,y)v_t(t,y)v_t(t,x)dydx  -\varepsilon\|v\|^2_{H^1_0}+\varepsilon\|v_t\|^2_{L^2}-\varepsilon\int_{\Omega}f_s(t, v)vdx \\
&\qquad +\varepsilon \int_{\Omega \times \Omega} K(x,y)v_t(t,y)v(t,x)dydx+\varepsilon \int_{\Omega}hvdx.
\end{aligned}
\end{equation*}

Reasoning as in \cite{ZhaoZhong2022} we can show that there exists a constant $K^\ast>0$ such that
\[
-k_s(t)\|v_t\|^{p+2}_{L^2} - \varepsilon k_s(t)\|v_t\|^p_{L^2}\int_\Omega v_tvdx \leqslant -k_0\|v_t\|^{p+2}_{L^2} \Big(1-\tfrac{\varepsilon K^\ast k_1}{k_0}\|v\|_{H^1_0}^{\frac{p}{p+1}}\Big) + \tfrac{\varepsilon}{12}\Big(1-\tfrac{\mu_0}{\lambda_1}\Big)\|v\|^2_{H^1_0},
\]
and we have the following:

\begin{lemma}\label{propa407}
If $\varepsilon_0:=\frac{\sqrt{\lambda_1}}{8}\left(1-\frac{\mu_0}{\lambda_1}\right)>0$, there exist constants $\delta_0,\delta_1>0$ such that for all $0<\varepsilon\leqslant \varepsilon_0$, $s\in \mathbb{R}$ and $t\geqslant 0$, we have   
\[
\frac{d}{dt}V_{\varepsilon}^s(t)\leqslant  -k_0\|v_t\|_{L^2}^{p+2}\left(\tfrac{1}{2}-
\delta_0\varepsilon(V_{\varepsilon}^s(t)+d_0)^{\frac{p}{2(p+1)}}\right) -\tfrac{4}{5}\left(1-\tfrac{\mu_0}{\lambda_1}\right)\varepsilon V_{\varepsilon}^s(t)+\delta_1.
\]
\end{lemma}
\begin{proof}
Since, from \ref{cond6}, $\int_\Omega \frac{\partial F_s}{\partial t}(t,x)dx \leqslant c_0$, for $0< \varepsilon \leqslant \varepsilon_0$, we obtain
\begin{equation}
\begin{aligned}\label{rimp0}
& \frac{d}{dt} V_{\varepsilon}^s(t)\leqslant -k_0\|v_t\|_{L^2}^{p+2}\big(1-\tfrac{\varepsilon K^\ast k_1}{k_0} \|v\|_{H^1_0}^{\frac{p}{p+1}}\big)+\tfrac{\varepsilon}{12}\left(1-\tfrac{\mu_0}{\lambda_1}\right)\|v\|_{H^1_0}^2+c_0\\ 
&\quad +K_0\|v_t\|^2_{L^2}-\varepsilon\|v\|^2_{H^1_0} +\varepsilon\|v_t\|^2_{L^2}-\varepsilon\int_{\Omega}F_s(t, v)dx - \varepsilon\tfrac{\lambda_1+\mu_0}{4}\|v\|^2_{L^2} -\varepsilon C_0  \\
&\quad +\tfrac{\varepsilon}{4}\left(\tfrac{3\mu_0}{\lambda_1}+1\right)\|v\|^2_{H^1_0} +\varepsilon g_0 + \tfrac{\varepsilon}{12}\left(1-\tfrac{\mu_0}{\lambda_1}\right)\|v\|_{H^1_0}^2 +\tfrac{3\varepsilon K_0^2}{\lambda_1-\mu_0}\|v_t\|^2_{L^2}\\
& \quad +\tfrac{\varepsilon}{16}\left(1-\tfrac{\mu_0}{\lambda_1}\right)\|v\|^2_{H^1_0}+\tfrac{4\varepsilon}{\lambda_1-\mu_0}h_0^2  \\
&\stackrel{(1)}{\leqslant} -k_0\|v_t\|_{L^2}^{p+2}\big(1-\tfrac{\varepsilon K^\ast k_1}{k_0}\|v\|_{H^1_0}^{\frac{p}{p+1}}\big)+C_1\|v_t\|^2_{L^2} - \tfrac{25\varepsilon}{48}\left(1-\tfrac{\mu_0}{\lambda_1}\right)\|v\|^2_{H^1_0} \\
&\quad -\tfrac{\varepsilon}{2}\left(1-\tfrac{\mu_0}{\lambda_1}\right)\|v_t\|^2_{L^2}-\varepsilon\int_{\Omega}F_s(t,v)dx-\varepsilon\tfrac{\lambda_1+\mu_0}{4}\|v\|^2_{L^2} +c_0 \\
& \quad -\varepsilon C_0 + \varepsilon_0 g_0 + \tfrac{4\varepsilon_0}{\lambda_1-\mu_0}h_0^2\\
&\stackrel{(2)}{\leqslant} -k_0\|v_t\|_{L^2}^{p+2}\big(1-\tfrac{\varepsilon K^\ast k_1}{k_0}\|v \|_{H^1_0}^{\frac{p}{p+1}}\big)+\tfrac{k_0}{2}\|v_t\|^{p+2}_{L^2}+C_2\\
&\quad -\tfrac{\varepsilon}{2}\left(1-\tfrac{\mu_0}{\lambda_1}\right)(\|v_t \|^2_{L^2} + \|v\|^2_{H^1_0})-\varepsilon\big[\int_{\Omega}F_s(t,v)dx+\tfrac{\lambda_1+\mu_0}{4}\|v\|_{L^2}^2 + C_0\big]\\
& \quad +c_0+\varepsilon_0g_0 + \tfrac{4\varepsilon_0}{\lambda_1-\mu_0}h_0^2 + C_2,
\end{aligned}
\end{equation}
where in $(1)$ we added and subtracted the term $\tfrac{\varepsilon}{2}(1-\tfrac{\mu_0}{\lambda_1})\|v_t\|^2_{L^2}$, used the estimate
\[
K_0\|v_t\|^2_{L^2}+\varepsilon\|v_t\|^2_{L^2}+\tfrac{3\varepsilon K_0^2}{\lambda_1-\mu_0}\|v_t\|^2_{L^2}+\tfrac{\varepsilon}{2}\left(1-\tfrac{\mu_0}{\lambda_1}\right)\|v_t\|^2_{L^2}\leqslant C_1\|v_t\|^2_{L^2},
\]
for $C_1:=K_0+\varepsilon_0+\tfrac{3\varepsilon_0 K_0^2}{\lambda_1-\mu_0}+\tfrac{\varepsilon_0}{2}\left(1-\tfrac{\mu_0}{\lambda_1}\right)$, and Poincar\'e's inequality. In $(2)$ we applied Young's inequality $ab\leqslant \frac{a^{q_1}}{q_1}+\frac{b^{q_2}}{q_2}$ with $q_1=\frac{p+2}{p}$ and $q_2=\frac{p+2}{p}$, noting that $\frac1{q_1}+\frac1{q_2}=1$, in order to obtain
\begin{align*}    
C_1\|v_t\|^2_{L^2}&=\left(\tfrac{k_0(p+2)}{4}\right)^{\frac{2}{p+2}}\|v_t\|^2_{L^2}\left(\tfrac{4}{k_0(p+2)}\right)^{\frac{2}{p+2}}C_1\\
&\leqslant \tfrac{k_0}{2}\|v_t\|^{p+2}_{L^2}+\tfrac{p}{p+2}\left(\tfrac{4}{k_0(p+2)}\right)^{\frac{2}{p}}C_1^{\frac{p+2}{p}} = \tfrac{k_0}{2}\|v_t\|^{p+2}_{L^2}+C_2,
\end{align*}
for $C_2:=\tfrac{p}{p+2}\left(\tfrac{4}{k_0(p+2)}\right)^{\frac{2}{p}}C_1^{\frac{p+2}{p}}$. From Proposition \ref{propa401} we have 
\[
2E_s(t) \leqslant \left(\tfrac{8\lambda_1}{\lambda_1-\mu_0}\right)(V^s_\varepsilon(t)+d_0),
\]
which implies 
\[
(2E_s(t))^{\frac{p}{2(p+1)}}\leqslant \left(\tfrac{8\lambda_1}{\lambda_1-\mu_0}\right)^{\frac{p}{2(p+1)}}(V_{\varepsilon}^s(t)+d_0)^{\frac{p}{2(p+1)}}.
\]
Hence
\begin{align*}
-k_0\|v_t\|^{p+2}_{L^2}&\big(\tfrac{1}{2}-\tfrac{\varepsilon K^\ast k_1}{k_0}(2E_s(t))^{\frac{p}{2(p+1)}}\big)\\
& \leqslant -k_0\|v_t\|_{L^2}^{p+2}\big(\tfrac{1}{2}-\tfrac{\varepsilon K^\ast k_1}{k_0}\left(\tfrac{8\lambda_1}{\lambda_1-\mu_0}\right)^{\frac{p}{2(p+1)}}(V_{\varepsilon}^s(t)+d_0)^{\frac{p}{2(p+1)}}\big),
\end{align*}
which, since \eqref{energyy} holds, we have that $\|v\|_{H^1_0}\leqslant (2E_s(t))^{\frac12}$, and then 
\begin{equation}
\begin{aligned}\label{rimp1}   
-k_0\|v_t\|_{L^2}^{p+2}& \left(\tfrac{1}{2}-\tfrac{\varepsilon K^\ast k_1}{k_0} \|v \|_{H^1_0}^{\frac{p}{p+1}}\right) \leqslant -k_0\|v_t\|_{L^2}^{p+2}\big(\tfrac{1}{2}-\tfrac{\varepsilon K^\ast k_1}{k_0}(2E_s(t))^{\frac{p}{2(p+1)}}\big)\\
&\leqslant -k_0\|v_t\|_{L^2}^{p+2}\big(\tfrac{1}{2}-\tfrac{\varepsilon K^\ast k_1}{k_0}\left(\tfrac{8\lambda_1}{\lambda_1-\mu_0}\right)^{\frac{p}{2(p+1)}}(V_{\varepsilon}^s(t)+d_0)^{\frac{p}{2(p+1)}}\big).
\end{aligned}
\end{equation}
Also observe that 
\begin{align*}
&\tfrac{1}{2}\left(1-\tfrac{\mu_0}{\lambda_1}\right)(\left\|v_t \right\|^2_{L^2}+\|v\|^2_{H^1_0})+\int_{\Omega}F_s(t,v)dx+\tfrac{\lambda_1+\mu_0}{4}\|v\|^2_{L^2} + C_0\\
&= \tfrac{1}{2}\left(1-\tfrac{\mu_0}{\lambda_1}\right)(\|v_t\|^2_{L^2}+\|v\|^2_{H^1_0})+E_s(t)-\tfrac{1}{2}(\|v_t\|^2_{L^2}+\|v\|^2_{H^1_0})\\
&=-\tfrac{\mu_0}{2\lambda_1}(\|v_t\|^2_{L^2}+\|v\|^2_{H^1_0})+E_s(t)\geqslant -\tfrac{\mu_0}{\lambda_1}E_s(t)+E_s(t)=E_s(t)\left(1-\tfrac{\mu_0}{\lambda_1}\right),
\end{align*}
implying that
\begin{equation}
\begin{aligned}\label{rimp2}
&-\tfrac{\varepsilon}{2}\left(1-\tfrac{\mu_0}{\lambda_1}\right)\left(\left\|v_t \right\|^2_{L^2}+\|v\|^2_{H^1_0}\right)-\varepsilon\left(\int_{\Omega}F_s(t,v)dx+\tfrac{\lambda_1+\mu_0}{4}\|v\|_{L^2}^2 + C_0\right)\\
&\leqslant -\left(1-\tfrac{\mu_0}{\lambda_1}\right)\varepsilon E_s(t)\stackrel{(1)}{\leqslant} -\tfrac{4}{5}\left(1-\tfrac{\mu_0}{\lambda_1}\right)\varepsilon V_{\varepsilon}^s(t) + \tfrac45\varepsilon_0 (1-\tfrac{\mu_0}{\lambda_1})d_0,
\end{aligned}
\end{equation}
where in $(1)$ we used Proposition \ref{propa401}. From \eqref{rimp0}, \eqref{rimp1} and \eqref{rimp2} we obtain
\begin{align*}
    \frac{d}{dt}V_{\varepsilon}^s(t)\leqslant -k_0& \|v_t\|_{L^2}^{p+2}\big(\tfrac{1}{2}-\tfrac{\varepsilon K^\ast k_1}{k_0}\left(\tfrac{8\lambda_1}{\lambda_1-\mu_0}\right)^{\frac{p}{2(p+1)}}(V_{\varepsilon}^s(t)+d_0)^{\frac{p}{2(p+1)}}\big)\\
    & -\tfrac{4}{5}\left(1-\tfrac{\mu_0}{\lambda_1}\right)\varepsilon V_{\varepsilon}^s(t)+ \tfrac45\varepsilon_0 (1-\tfrac{\mu_0}{\lambda_1})d_0 + c_0+\varepsilon_0g_0 + \tfrac{8\varepsilon_0}{\lambda_1-\mu_0}h_0 + C_2,
\end{align*}
and the result is proved, choosing
\[
\delta_0:= \tfrac{K^\ast k_1}{k_0}\left(\tfrac{8\lambda_1}{\lambda_1-\mu_0}\right)^{\frac{p}{2(p+1)}}
\]
and
\[
\delta_1:=\tfrac45\varepsilon_0 (1-\tfrac{\mu_0}{\lambda_1})d_0 + c_0+ \varepsilon_0g_0 + \tfrac{8\varepsilon_0}{\lambda_1-\mu_0}h_0 + C_2.
\]
\end{proof}

As a consequence of the previous result, we have the following lemma, which tell us in general terms that a suitable estimate for the function $V_{\varepsilon}^s$ at a time $\tau$ gives us a similar estimate for $V_{\varepsilon}^s$ for all $t\geqslant \tau$. 

\begin{lemma}\label{vale}
Assume that $0<\varepsilon \leqslant \varepsilon_0$ and $C>0$. If for $\tau\geqslant 0$ we have
\[
V_{\varepsilon}^s(\tau)\leqslant (2\delta_0\varepsilon)^{-\frac{2(p+1)}{p}}-\tfrac{5}{4}\tfrac{\lambda_1}{\lambda_1-\mu_0}\delta_1\varepsilon^{-1}-d_0,
\]
then for all $t\geqslant \tau$ we have
\[
V_{\varepsilon}^s(t)\leqslant (2\delta_0\varepsilon)^{-\frac{2(p+1)}{p}}-d_0.
\]
\end{lemma}
\begin{proof}
We fix $0<\varepsilon \leqslant \varepsilon_0$ and set
\begin{align}
& \Psi(\varepsilon):=(2\delta_0\varepsilon)^{-\frac{2(p+1)}{p}}-d_0, \quad \hbox{ and } \label{def.Psi} \\
& \Phi(\varepsilon):=(2\delta_0\varepsilon)^{-\frac{2(p+1)}{p}}-\tfrac{5}{4}\tfrac{\lambda_1}{\lambda_1-\mu_0}\delta_1\varepsilon^{-1}-d_0. \label{def.Phi}
\end{align}
Note that $V_\varepsilon^s(\tau)\leqslant \Phi(\varepsilon) < \Psi(\varepsilon)$. By the continuity of $[0,\infty)\ni t \mapsto V_{\varepsilon}^s(t)$, there exists $T>\tau$ such that $V_{\varepsilon}^s(t)\leqslant \Psi(\varepsilon)$ for $\tau\leqslant t < T$. Let $T_0=\inf\{t \colon t\geqslant \tau \hbox{ and } V_{\varepsilon}^s(t)> \Psi\}$ and assume that $T_0<\infty$. It is clear that $T_0\geqslant T> \tau$, $V_{\varepsilon}^s(t)\leqslant \Psi(\varepsilon)$ for $\tau\leqslant t \leqslant T_0$ and $V_{\varepsilon}^s(T_0)=\Psi(\varepsilon)$.

From Proposition \ref{propa401} we know that $V_{\varepsilon}^s(t)+d_0\geqslant \frac{1}{4}\left(1-\frac{\mu_0}{\lambda_1}\right)E_s(t)\geqslant 0$. Then, for $\tau \leqslant t\leqslant T_0$ we have
\[
(V_{\varepsilon}^s(t)+d_0)^{\frac{p}{2(p+1)}} \leqslant (\Phi(\varepsilon)+d_0)^{\frac{p}{2(p+1)}} \leqslant \frac{1}{2\delta_0\varepsilon},
\]
which implies that
\begin{equation}\label{niceresult}
 \tfrac{1}{2}-\delta_0\varepsilon(V_{\varepsilon}^s(t)+d_0)^{\frac{p}{2(p+1)}}\geqslant 0.
\end{equation}
Using Lemma \ref{propa407}, for $\tau\leqslant t\leqslant T_0$ and $0<\varepsilon\leqslant \varepsilon_0$, we obtain 
\[
\frac{d}{dt}V_{\varepsilon}^s(t)+\tfrac{4}{5}\left(1-\tfrac{\mu_0}{\lambda_1}\right)\varepsilon V_{\varepsilon}^s(t)\leqslant \delta_1.
\]
Using Gronwall's inequality, for $\tau \leqslant t\leqslant T_0$ it follows that
\begin{align*}
V_{\varepsilon}^s(t) & \leqslant V_{\varepsilon}^s(\tau)e^{-\frac{4}{5}\left(1-\frac{\mu_0}{\lambda_1}\right)\varepsilon(t-\tau)}+\tfrac{5}{4}\tfrac{\lambda_1}{\lambda_1-\mu_0}\delta_1\varepsilon^{-1} < V_\varepsilon^s(\tau) + \tfrac{5}{4}\tfrac{\lambda_1}{\lambda_1-\mu_0}\delta_1\varepsilon^{-1}\\
& \leqslant \Phi(\varepsilon) + \tfrac{5}{4}\tfrac{\lambda_1}{\lambda_1-\mu_0}\delta_1\varepsilon^{-1} = \Psi(\varepsilon).
\end{align*}
Taking $t=T_0$ we obtain $V_\varepsilon^s(T_0)<\Psi(\varepsilon)$, which is a contradiction. Thus $T_0=\infty$ and the proof is complete.
\end{proof}

\begin{remark}\label{rem.niceresult}
One aspect that we want to draw attention to is that we concluded that if $V_{\varepsilon}^s(\tau)\leqslant \Phi(\varepsilon)$ then \eqref{niceresult} holds for all for all $t\geqslant \tau$. 
\end{remark}

Now we can prove Proposition \ref{EsC}.
 
\begin{proof}[Proof of Proposition \ref{EsC}]
From Proposition \ref{funcFimp} and the continuous inclusion $H^1_0(\Omega)\hookrightarrow L^{4}(\Omega)$, with constant $c>0$, we obtain 
\begin{align*}
\left|\int_{\Omega}F(s, v) dx\right|& \leqslant \int_{\Omega}|F(s, v)| dx \leqslant \int_{\Omega}4c_0(1+|v|^{4}) dx\\
& =4c_0|\Omega| + 4c_0\|v\|_{L^{4}}^{4} \leqslant 4c_0|\Omega|+4c_0c^{4}\|v\|_{H^1_0}^{4} \leqslant \alpha_0(1+\left\|v\right\|^{4}_{H^1_0}),    
\end{align*}
where $\alpha_0=4c_0\max\{|\Omega|,c^4\}$. Fix $R>0$ and let $V_0=(u_0,u_1)\in X$ be such that $V_0\in \overline{B}^X_R$. From the inequality above and Poincar\'e's inequality, we deduce 
\begin{align*}
E_s(0)&=\tfrac{1}{2}\left(\left\|u_1\right\|_{L^2}^2+\| u_0\|_{H^1_0}^2\right)+\int_{\Omega}F(s, u_0)dx+\tfrac{\lambda_1+\mu_0}{4}\|u_0\|^2_{L^2} +C_0\\
&\leqslant \tfrac{1}{2}\|V_0\|^2_X+\alpha_0(1+\| u_0\|^{4}_{H^1_0})+\tfrac{\lambda_1+\mu_0}{4\lambda_1}\|u_0\|^2_{H^1_0}+C_0\\
&\leqslant \tfrac{1}{2}R^2+\alpha_0R^{4}+\tfrac{\lambda_1+\mu_0}{4\lambda_1}R^2+\alpha_0+C_0.
\end{align*}
By Proposition \ref{propa401} we obtain
\begin{align}\label{cr2cr1}
    V_{\varepsilon}^s(0)\leqslant \tfrac{5}{4}E_s(0)+d_0\leqslant \tfrac{5}{4}\left(\tfrac{1}{2}R^2+\alpha_0 R^{4}+\tfrac{\lambda_1+\mu_0}{4\lambda_1}R^2+\alpha_0+C_0\right)+d_0:=\gamma_R.
\end{align}
Consider the maps $\Psi,\Phi\colon (0,\infty)\to \mathbb{R}$ defined in \eqref{def.Psi} and \eqref{def.Phi}, respectively. Note that
\begin{align}\label{derivativeofphi}
\Phi'(\varepsilon)=-\frac{1}{\varepsilon^2}\left(\tfrac{2(p+1)}{p}(2\delta_0)^{\frac{-2(p+1)}{p}}\varepsilon^{-\frac{p+2}{p}}-\tfrac{5}{4}\tfrac{\lambda_1}{\lambda_1-\mu_0}\delta_1\right).
\end{align}
Setting 
\[
\varepsilon_1 : = \left(\tfrac{5p}{8(p+1)}\tfrac{\lambda_1}{\lambda_1-\mu_0}\delta_1(2\delta_0)^{\frac{2(p+1)}{p}}\right)^{-\frac{p}{p+2}},
\]
we can see that $\varepsilon_1>0$, that $\Phi'(\varepsilon)<0$ for $0<\varepsilon<\varepsilon_1$ and $\Phi'(\varepsilon)>0$ for $\varepsilon>\varepsilon_1$. Hence $\Phi$ is strictly decreasing in the interval $(0,\varepsilon_1]$ and strictly increasing in $[\varepsilon_1,\infty)$. It is clear that $\lim\limits_{\varepsilon\to\infty}\Phi(\varepsilon)=-d_0$. On the other hand, since $\frac{2(p+1)}{p}>1$, we have $\lim\limits_{\varepsilon\to 0^+}\Phi(\varepsilon)=\infty$. Hence $\Phi(\varepsilon_1)<-d_0<0$ and, from the Intermediate Value Theorem and the monotonicity of $\Phi$, there exists a unique point $0<\varepsilon_2<\varepsilon_1$ such that $\Phi(\varepsilon_2)=0$ and $\Phi(\varepsilon)>0$ for all $0<\varepsilon <\varepsilon_2$. Observe that $\Phi\colon \left(0, \varepsilon_2\right]\rightarrow \left[0, \infty\right)$ is bijective and take $\alpha_R := \min\{\varepsilon_0,  \Phi^{-1}(\gamma_R)\}\leqslant \varepsilon_2$. Since $\alpha_R\leqslant \Phi^{-1}(\gamma_R)$ and $\Phi$ is decreasing in $(0, \varepsilon_2]$, it follows from \eqref{cr2cr1} that $V_\varepsilon^s(0)\leqslant \gamma_R \leqslant \Phi(\alpha_R)$. Thus, from Lemma \ref{vale} we obtain $V_{\varepsilon}^s(t)\leqslant \Psi(\alpha_R)$ for all $t\geqslant 0$. From this fact and Proposition \ref{propa401}, for all $t\geqslant 0$ we obtain 
\begin{align*}
E_s(t)&\leqslant \tfrac{4\lambda_1}{\lambda_1-\mu_0}(V_{\varepsilon}^s(t)+d_0)\leqslant \tfrac{4\lambda_1}{\lambda_1-\mu_0}(\Psi(\alpha_R)+d_0),
\end{align*}
which completes the proof, defining $c_R:= \tfrac{4\lambda_1}{\lambda_1-\mu_0}(\Psi(\alpha_R)+d_0)$.
\end{proof}

\medskip \noindent \textsc{Development of (III).}

Continuing our scheme, our goal for now is to prove the following result.

\begin{proposition}\label{estimate.taumax}
There exists a constant $R_0>0$ such that for any $R>0$ and $s\in \mathbb{R}$,
\[  \limsup\limits_{\tau\rightarrow\infty}\Big(\sup_{\|V_0\|_X\leqslant R}E_s(\tau)\Big)\leqslant R_0.
\]
\end{proposition}
\begin{proof}
Consider the function $\Theta\colon [0,\infty)\to (0,\infty)$ defined by
\[
\Theta(\sigma)=\min\left\{\varepsilon_0, \Phi^{-1}\left(\tfrac{5}{4}\sigma+d_0\right)\right\} \quad \hbox{ for } \sigma \geqslant 0.
\]
It is clear that $\Theta$ is a continuous and strictly decreasing function. Now fix $r\geqslant 0$ and define $\varepsilon:=\Theta(E_s(r))$. Note that $0<\varepsilon \leqslant \varepsilon_0$, $\varepsilon<\varepsilon_2$ and
$
\varepsilon \leqslant \Phi^{-1}\left(\tfrac{5}{4}E_s(r)+d_0\right).
$
Since $\Phi$ is decreasing in $(0,\varepsilon_2]$, we have $\Phi(\varepsilon)\geqslant \frac{5}{4}E_s(r)+d_0\geqslant V_{\varepsilon}^s(r)$. From Remark \ref{rem.niceresult}, for all $t\geqslant r$ we obtain
\[
\tfrac{1}{2}-\delta_0\varepsilon(V_{\varepsilon}^s(t)+d_0)^{\frac{p}{2(p+1)}}\geqslant 0.
\]
This information and Lemma \ref{propa407}, for $t\geqslant r$, yield
\begin{align*}
\frac{d}{dt}V_{\varepsilon}^s(t)&\leqslant  -k_0\|v_t\|_{L^2}^{p+2}\left(\tfrac{1}{2}-\varepsilon\delta_0(V_{\varepsilon}^s(t)+d_0)^{\frac{p}{2(p+1)}}\right) -\tfrac{4}{5}\left(1-\tfrac{\mu_0}{\lambda_1}\right)\varepsilon V_{\varepsilon}^s(t)+\delta_1\\
&\leqslant -\tfrac{4}{5}\left(1-\tfrac{\mu_0}{\lambda_1}\right)\Theta(E_s(r)) V_{\varepsilon}^s(t)+\delta_1.
\end{align*}
Applying Gr\"onwall's inequality we obtain
\begin{align*}
V_{\varepsilon}^s(t)\leqslant V_{\varepsilon}^s(r)e^{-\frac{4}{5}\left(1-\frac{\mu_0}{\lambda_1}\right)\Theta(E_s(r))(t-r)}+\tfrac{5}{4}\tfrac{\lambda_1\delta_1}{\lambda_1-\mu_0}\left[\Theta(E_s(r))\right]^{-1},
\end{align*}
for $t\geqslant r$. Using Proposition \ref{propa401} we obtain
\begin{align*}
\tfrac{1}{4}& \left(1-\tfrac{\mu_0}{\lambda_1}\right)E_s(t)-d_0 \leqslant  V_{\varepsilon}^s(r)
e^{-\frac{4}{5}\left(1-\frac{\mu_0}{\lambda_1}\right)\Theta(E_s(r))(t-r)}+\tfrac{5}{4}\tfrac{\lambda_1\delta_1}{\lambda_1-\mu_0}\left[\Theta(E_s(r))\right]^{-1}\\
&\leqslant \left(\tfrac{5}{4}E_s(r)+d_0\right)e^{-\frac{4}{5}\left(1-\frac{\mu_0}{\lambda_1}\right)\Theta(c_R)(t-r)}+\tfrac{5}{4}\tfrac{\lambda_1\delta_1}{\lambda_1-\mu_0}\left[\Theta(E_s(r))\right]^{-1},
\end{align*}
where we have used that, for a fixed $R>0$, since $E_s(r)\leqslant c_R$ for all $r\geqslant 0$ and $\Theta$ is decreasing, $\Theta(E_s(r))\geqslant \Theta(c_R)>0$. Hence, since $\Theta$ is decreasing, for this given $R>0$, setting $\gamma_{R,s}(t):=\sup\limits_{\|V_0\|_X\leqslant R}E_s(t)$ we obtain
\begin{align*}
 \tfrac{1}{4}\left(1\!-\!\tfrac{\mu_0}{\lambda_1}\right)\gamma_{R,s}(t)-d_0 \leqslant \left(\tfrac{5}{4}\gamma_{R,s}(r)+d_0\right)e^{-\frac{4}{5}\left(1-\frac{\mu_0}{\lambda_1}\right)\Theta(c_R)(t-r)} +\tfrac{5}{4}\tfrac{\lambda_1\delta_1}{\lambda_1-\mu_0}\left[\Theta\left(\gamma_{R,s}(r)\right)\right]^{-1}.   
\end{align*}
Consequently, if $w_{R,s}:=\limsup\limits_{t\to \infty}\gamma_{R,s}(t)$, we obtain
\[
\tfrac{1}{4}\left(1-\tfrac{\mu_0}{\lambda_1}\right) w_{R,s}
\leqslant \tfrac{5}{4}\tfrac{\lambda_1\delta_1}{\lambda_1-\mu_0}\left[\Theta\left(w_{R,s}\right)\right]^{-1}+d_0,
\]
that is,
\begin{align}\label{inn54}
\frac{1}{w_{R,s}}\left(\tfrac{5}{4d_0}\tfrac{\lambda_1}{\lambda_1-\mu_0}\left[\Theta(w_{R,s})\right]^{-1}+1\right)\geqslant \tfrac{1}{4d_0}\left(1-\tfrac{\mu_0}{\lambda_1}\right).
\end{align}
Now, for $z>0$, consider the function
\begin{align*}
    G(z)=\frac{1}{z}\left(\tfrac{5}{4d_0}\tfrac{\lambda_1}{\lambda_1-\mu_0}\left[\Theta(z)\right]^{-1}+1\right).
\end{align*}
Since $\Phi^{-1}$ is a positive decreasing function approaching to zero at infinity, we have $\lim\limits_{z\rightarrow \infty}\Theta(z)=\lim\limits_{z\rightarrow \infty}\min\left\{\varepsilon_0, \Phi^{-1}(\frac{5}{4}z+d_0)\right\}=\lim\limits_{z\rightarrow \infty}\Phi^{-1}(\frac{5}{4}z+d_0)=0$, and thus
\begin{align*}
\lim\limits_{z\rightarrow \infty}G(z)&=\lim\limits_{z\rightarrow \infty}\frac{\frac{5}{4d_0}\left(1-\frac{\mu_0}{\lambda_1}\right)^{-1}\left[\Phi^{-1}(\frac{5}{4}z+d_0)\right]^{-1}+1}{z}\\
&=\lim\limits_{\sigma\rightarrow \infty}\frac{\frac{5}{4d_0}\left(1-\frac{\mu_0}{\lambda_1}\right)^{-1}\sigma+1}{\frac{4}{5}\left(\Phi\left(\frac{1}{\sigma}\right)-d_0\right)} \stackrel{(1)}{=}\lim\limits_{\sigma\rightarrow \infty}\frac{\frac{5}{4d_0}\left(1-\frac{\mu_0}{\lambda_1}\right)^{-1}}{-\frac{4}{5}\Phi '\left(\frac{1}{\sigma}\right)\frac{1}{\sigma^2}}\stackrel{(2)}{=}0,
\end{align*}
where in (1) we applied the L'Hospital Rule and for (2) we note that
\begin{align*}    
\lim\limits_{\sigma\rightarrow\infty}-\tfrac{4}{5}\Phi'\left(\tfrac{1}{\sigma}\right)\tfrac{1}{\sigma^2}=\lim\limits_{\sigma\rightarrow\infty}\tfrac{4}{5}\left[\tfrac{2(p+1)}{p}(2\delta_0)^{\frac{-2(p+1)}{p}}\sigma^{\frac{p+2}{p}}-\tfrac{5}{4}\tfrac{\lambda_1\delta_1}{\lambda_1-\mu_0}\right]=\infty.
\end{align*}
 Since $G(w_{R,s})\geqslant \frac{1}{4d_0}\left(1-\frac{\mu_0}{\lambda_1}\right)$, this implies that there exists $R_0>0$ such that $w_{R,s}\leqslant R_0$ for all $R>0$, since the constant $\frac{1}{4d_0}(1-\frac{\mu_0}{\lambda_1})$ does not depend on $R$.
\end{proof}

\medskip \noindent \textsc{Development of (IV).}

With all the work done in items (I), (II) and (III), now the proof of Theorem \ref{teo.existence.paf} follows easily.

\begin{proof}[Proof of Theorem \ref{teo.existence.paf}]
Let $R>0$. From Proposition \ref{estimate.taumax} there exists $\tau_0=\tau_0(R)\geqslant 0$ such that for $\tau\geqslant \tau_0$ we have $\sup\limits_{\|(u_0,u_1)\|_X\leqslant R}E_{t-\tau}(\tau)\leqslant 2R_0$ for all $t\in \mathbb{R}$. Therefore, if $\tau\geqslant \tau_0$ and $V_0=(u_0,u_1)$, it follows that
\begin{align*}
\left\|S(t,t-\tau)(u_0,u_1)\right\|_X^2&=\|V(\tau,V_0)\|^2_X \leqslant 2E_{t-\tau}(\tau) \leqslant 4R_0,
\end{align*}
and the proof is complete, taking $r_0:=2R_0^\frac12$.
\end{proof}

This implies, in particular, that the family $\hat{B}=\{B_t\}_{t\in \mathbb{R}}$, with $B_t=\overline{B}_{r_0}^X$ for all $t\in \mathbb{R}$, is a uniformly pullback absorbing family for $S$.

\bigskip \noindent \textbf{Existence of a closed, uniformly bounded and positively invariant uniformly pullback absorbing family}\label{obtainingC}\medskip

Using the family $\hat{B}$ presented above, we will construct a family $\hat{C}$ that is closed, uniformly bounded, positively invariant and uniformly pullback absorbing for $S$. Since $\overline{B}^X_{r_0}$ is bounded, there exists $\tau_1\geqslant 0$ such that $S(t,t-\tau)\overline{B}^X_{r_0}\subset \overline{B}^X_{r_0}$ for all $\tau\geqslant \tau_1$ and $t\in\mathbb{R}$. Consider the family $\hat{C}=\left\{C_t\right\}_{t\in\mathbb{R}}$ defined by 
\begin{align*}
    C_t=\overline{\bigcup_{\tau\geqslant \tau_1}S(t,t-\tau)\overline{B}_{r_0}^X} \quad \hbox{ for each } t\in \mathbb{R}.
\end{align*}

\begin{theorem}
    The family $\hat{C}$ is closed, bounded, positively invariant and uniformly pullback absorbing for $S$.
\end{theorem}
\begin{proof}
Clearly $\hat{C}$ is closed, and $C_t \subset \overline{B}_{r_0}^X$ for all $t\in\mathbb{R}$, which implies $\bigcup_{t\in \mathbb{R}}C_t \subset \overline{B}_{r_0}^X$, that is, $\hat{C}$ is uniformly bounded. Now, let $t\geqslant s$. Then we can write $s=t-\sigma$ for some $\sigma\geqslant 0$. Note that
\begin{align*}   
S(t,s)C_s&=S(t,t-\sigma)C_{t-\sigma}=S(t,t-\sigma)\overline{\bigcup_{\tau\geqslant \tau_1}S(t-\sigma,t-\sigma-\tau)\overline{B}_{r_0}^X} \\
&\subset \overline{S(t,t-\sigma)\bigcup_{\tau\geqslant \tau_1}S(t-\sigma,t-\sigma-\tau)\overline{B}_{r_0}^X}=\overline{\bigcup_{\tau\geqslant \tau_1}S(t,t-\sigma)S(t-\sigma,t-\sigma-\tau)\overline{B}_{r_0}^X}\\
&=\overline{\bigcup_{\tau\geqslant t_1}S(t,t-(\sigma+\tau))\overline{B}_{r_0}^X} \subset \overline{\bigcup_{\tau\geqslant \sigma+\tau_1}S(t,t-\tau)\overline{B}_{r_0}^X}\subset \overline{\bigcup_{\tau\geqslant t_1}S(t,t-\tau)\overline{B}_{r_0}^X}=C_t,
\end{align*}
which proves the positively invariance of the family $\hat{C}$.

To see that $\hat{C}$ is uniformly pullback absorbing, let $D\subset H^1_0(\Omega)\times L^2(\Omega)$ be a bounded set. We know that there exist $\tau_0, \tau_1\geqslant 0$ such that $S(t,t-\tau)D\subset \overline{B}_{r_0}^X$ and $S(t,t-\tau)\overline{B}_{r_0}^X\subset \overline{B}_{r_0}^X$ for every $t\in\mathbb{R}$ and $\tau\geqslant \max\{\tau_0,\tau_1\}$. If $t\in \mathbb{R}$ and $\tau \geqslant \tau_0+\tau_1$ then $\tau-\tau_1\geqslant \tau_0$ and we obtain
\begin{align*}
S(t,t-\tau)D &= S(t,t-\tau_1)S(t-\tau_1,t-\tau)D\\
&=S(t,t-\tau_1)S(t-\tau_1,t-\tau_1-(\tau-\tau_1))D \subset S(t,t-\tau_1)\overline{B}_{r_0}^X\subset C_t,
\end{align*}
and the proof is complete.
\end{proof}

\subsection{Existence of a generalized polynomial pullback attractor}

Consider the closed, uniformly bounded, positively invariant and uniformly pullback absorbing family $\hat{C}$ obtained in Subsection \ref{obtainingC}. If $s\in \mathbb{R}$ and $V_1,V_2\in C_s\subset \overline{B}^X_{r_0}$, from Proposition \ref{EsC} there exists a constant $c_{r_0}>0$ such that 
\[
\|V(t,V_1)\|_X \leqslant c_{r_0} \quad \hbox{ and } \quad \|V(t,V_2)\|_X \leqslant c_{r_0} \quad \hbox{ for all } t\geqslant 0.
\]
If $V(t,V_1)=(v(t),v_t(t))$ and $V(t,V_2)=(w(t),w_t(t))$, Setting $Z(t)=(z(t),z_t(t))=V(t,V_1)-V(t,V_2)$ we have $z=v-w$, and \eqref{fghjmmm} holds. Defining 
\[
\mathcal{E}_s(t)=\frac12\|Z(t)\|^2_X = \frac12\left(\|z\|^2_{H^1_0}+\|z_t\|^2_{L^2}\right),
\]
for $T>0$ we have
\begin{equation}\label{centralx}
\begin{aligned}
T&\mathcal{E}_s(T)=-\frac12 \langle z_t,z\rangle\Big|_0^T+\int_0^T\|z_t\|^2_{L^2}dt-\frac12 \int_0^T k_s(t)\langle \|v_t\|^p_{L^2} v_t-\|w_t\|^p_{L^2} w_t, z\rangle dt\\
&-\frac12 \int_0^T\langle f_s(t,v)-f_s(t,w), z\rangle dt
+\frac12 \int_0^T\int_{\Omega\times\Omega} K(x,y)z_t(t,y)z(t,x) dydxdt\\
&-\int_0^T\int_t^Tk_s(\tau)\langle \|v_t\|^p_{L^2} v_t-\|w_t\|^p_{L^2} w_t, z_t \rangle d\tau dt -\int_0^T\int_t^T \langle f_s(\tau, v)-f_s(\tau, w), z_t\rangle d\tau dt\\
&+\int_0^T\int_t^T\int_{\Omega\times\Omega}K(x,y)z_t(\tau,y)z_t(\tau,x) dydxd\tau dt.
\end{aligned}
\end{equation}    

\begin{proposition}\label{prop.ghxxgh}
For $T>0$ there exist constants $\Gamma_{T,1},\Gamma_{T,2}>0$ such that
\begin{equation}\label{ghxxgh}
\begin{aligned}
&\mathcal{E}_s(T)\leqslant  \Gamma_{T,1}\sup_{t\in [0,T]}\|z(t)\|_{L^2}+ \Gamma_{T,2}\Big(\mathcal{E}_s(0)-\mathcal{E}_s(T)\\
&\qquad +\left|\int_0^T \langle f_s(\tau, v)-f_s(\tau, w), z_t\rangle d\tau\right|+ 2c_{r_0}\int_0^T \left\|\int_{\Omega}K(x,y)z_t(t,y) dy\right\|_{L^2} dt\Big)^{\frac{2}{p+2}}\\   
&\qquad +\tfrac1T\left|\int_0^T\int_t^T \langle f_s(\tau, w)-f_s(\tau, v), z_t\rangle d\tau dt\right|+2c_{r_0}\int_0^T \left\|\int_{\Omega}K(x,y)z_t(t,y) dy\right\|_{L^2} dt.
\end{aligned}
\end{equation}
\end{proposition}
\begin{proof}
Using Proposition \ref{prop.Cp} we have
\begin{align}\label{cpneto}
    \langle \|v_t\|^p_{L^2}v_t-\|w_t\|^p_{L^2}w_t, z_t\rangle \geqslant 2^{-p}\|z_t\|_{L^2}^{p+2},
\end{align}
which implies that
$
\|z_t\|_{L^2}^2\leqslant 2^{\frac{2p}{p+2}}\langle\|v_t\|^p_{L^2}v_t-\|w_t\|^p_{L^2}w_t, z_t\rangle^{\frac{2}{p+2}}.
$
Using the fact that $[0,\infty)\ni r\mapsto r^{\frac{2}{2+p}}$ is a continuous concave function, \cite[Proposition 1.2.11]{Tuomas} implies that
\begin{equation}\label{ineqzt2}
\begin{aligned}    
\int_0^T\|z_t\|_{L^2}^2 dt & \leqslant 2^{\frac{2p}{p+2}}\int_0^T \langle \|v_t\|^p_{L^2}v_t-\|w_t\|^p_{L^2}w_t, z_t\rangle^{\frac{2}{p+2}} dt\\
&\leqslant (4T)^{\frac{p}{p+2}}\left(\int_0^T \langle \|v_t\|^p_{L^2}v_t-\|w_t\|^p_{L^2}w_t, z_t\rangle dt\right)^{\frac{2}{p+2}} \\
&\stackrel{(\ast)}{\leqslant} \left(\tfrac{4T}{k_0}\right)^{\frac{p}{p+2}}\left(\int_0^Tk_s(t) \langle \|v_t\|^p_{L^2}v_t-\|w_t\|^p_{L^2}w_t, z_t\rangle dt\right)^{\frac{2}{p+2}},
\end{aligned}
\end{equation}
where in $(\ast)$ we used the fact that $\frac{k_s(t)}{k_0}\geqslant 1$ for all $t,s\in \mathbb{R}$.

Now, using \eqref{cpneto} , we obtain
\begin{equation}\label{putinto}  
\begin{aligned}  
0\leqslant & \int_0^Tk_s(t)\langle \|v_t\|^p_{L^2}v_t-\|w_t\|^p_{L^2}w_t,z_t \rangle dt =\mathcal{E}_s(0)-\mathcal{E}_s(T)\\
& \quad -\int_0^T \langle f_s(\tau, v)-f_s(\tau, w), z_t\rangle d\tau +  \int_0^T\int_{\Omega\times\Omega} K(x,y)z_t(\tau,y)z_t(\tau,x) dydx dt\\
&\leqslant \mathcal{E}_s(0)-\mathcal{E}_s(T)+\left|\int_0^T \langle f_s(\tau, v)-f_s(\tau,w), z_t\rangle d\tau\right|\\
& \qquad + \int_0^T\int_{\Omega\times\Omega} K(x,y)z_t(\tau,y)z_t(\tau,x) dydx dt.
\end{aligned}
\end{equation}
Since $\|z_t\|_{L^2}\leqslant 2c_{r_0}$ for all $t\geqslant 0$, using H\"older's inequality we have
\begin{align*}
\int_0^T\int_{\Omega\times\Omega}K(x,y)z_t(\tau,y)z_t(\tau,x) dydx d\tau & \leqslant \int_0^T \left\|\int_{\Omega}K(x,y)z_t(\tau,y) dy\right\|_{L^2}\|z_t(\tau,x)\|_{L^2} d\tau\\
&\leqslant 2c_{r_0}\int_0^T \left\|\int_{\Omega}K(x,y)z_t(t,y) dy\right\|_{L^2} dt,
\end{align*}
plugging \eqref{putinto} into \eqref{ineqzt2} we obtain
\begin{equation}\label{inx0}
\begin{aligned}
\int_0^T\left\|z_t\right\|_{L^2}^2 dt \leqslant \left(\tfrac{4T}{k_0}\right)^{\frac{p}{p+2}}\Big( & \mathcal{E}_s(0)  -\mathcal{E}_s(T) + \left|\int_0^T \langle f_s(\tau, v)-f_s(\tau, w), z_t\rangle d\tau\right|\\
& + 2c_{r_0}\int_0^T \left\|\int_{\Omega}K(x,y)z_t(t,y) dy\right\|_{L^2} dt\Big)^{\frac{2}{p+2}}
\end{aligned}
\end{equation}
Furthermore
\begin{equation*}
-\tfrac12 \langle z_t,z\rangle\Big|_0^T\leqslant  2c_{r_0}\sup_{t\in [0,T]}\left\|z(t)\right\|_{L^2},
\end{equation*}
and
\begin{equation*}
\begin{aligned}
-\tfrac12&\int_0^T k_s(t)\langle \|v_t\|^p_{L^2}v_t-\|w_t\|^p_{L^2}w_t, z\rangle dt \leqslant \tfrac{k_1}{2}\int_0^T | \|v_t\|^p_{L^2}v_t-\|w_t\|^p{L^2}w_t, z\rangle| dt\\
&\leqslant \tfrac{k_1}{2}\int_0^T \big\|\|v_t\|^p_{L^2}v_t-\|w_t\|^p_{L^2}w_t\big\|_{L^2}\|z(t)\|_{L^2} dt \leqslant k_1c_{r_0}^{p+2}T\sup_{t\in[0,T]}\left\|z(t)\right\|_{L^2}.
\end{aligned}
\end{equation*}
Using Lemma \ref{wvtlema} we have
\begin{equation*}
\begin{aligned}
-\tfrac12 & \int_0^T\langle f_s(t,v)-f_s(t,w), z\rangle dt \leqslant \tfrac12 \int_0^T \|f_s(t,v)-f_s(t,w)\|_{L^2}\|z(t)\|_{L^2} dt\\
&\leqslant \tfrac{L_0(1+2c_{r_{0}}^{2})}{2} \int_0^T \|z(t)\|_{H^1_0}\|z(t)\|_{L^2} dt \leqslant  \tfrac{L_0c_{r_0}(1+2c_{r_{0}}^{2})T}{2} \sup_{t\in[0,T]}\|z(t)\|_{L^2}.
\end{aligned}
\end{equation*}
Using H\"older's inequality we obtain
\begin{equation*}
\begin{aligned}
\tfrac12 \int_0^T\int_{\Omega\times\Omega}& K(x,y)z_t(t,y)z(t,x) dydxdt \leqslant \tfrac{K_0}{2} \int_0^T \|z_t\|_{L^2}\|z\|_{L^2}dt \leqslant K_0c_{r_0}T\sup_{t\in [0,T]}\|z(t)\|_{L^2},
\end{aligned}
\end{equation*}
and
\begin{align*}
\int_0^T\int_t^T &  \int_{\Omega\times\Omega}K(x,y)z_t(\tau,y)z_t(\tau,x) dydx d\tau dt\\
&\leqslant \int_0^T\int_t^T \left\|\int_{\Omega}K(x,y)z_t(\tau,y) dy\right\|_{L^2}\left\|z_t(\tau,y)\right\|_{L^2}d\tau dt\leqslant 2c_{r_0}T\int_0^T \left\|\int_{\Omega}K(x,y)z_t(t,y) dy\right\|_{L^2} dt
\end{align*}
Using all these inequalities in \eqref{centralx}, we conclude that
\begin{align*}
&T\mathcal{E}_s(T)\leqslant 2c_{r_0}\sup_{t\in [0,T]}\left\|z(t)\right\|_{L^2}+\left(\tfrac{4T}{k_0}\right)^{\frac{p}{p+2}}\Big(\mathcal{E}_s(0)  -\mathcal{E}_s(T) \\
& \qquad + \left|\int_0^T \langle f_s(\tau, v)-f_s(\tau, w), z_t\rangle d\tau\right| + 2c_{r_0}\int_0^T \left\|\int_{\Omega}K(x,y)z_t(t,y) dy\right\|_{L^2} dt\Big)^{\frac{2}{p+2}} \\
&\quad + k_1c_{r_0}^{p+2}T\sup_{t\in[0,T]}\|z(t)\|_{L^2} +\tfrac{L_0c_{r_0}(1+2cr_{0}^2)T}{2}\sup_{t\in[0,T]}\|z(t)\|_{L^2}+K_0c_{r_0}T \sup_{t\in [0,T]}\|z(t)\|_{L^2}\\
&\quad  +\left|\int_0^T\int_t^T \langle f_s(\tau, v)-f_s(\tau, w), z_t\rangle d\tau dt\right|+2c_{r_0}T\int_0^T \left\|\int_{\Omega}K(x,y)z_t(t,y) dy\right\|_{L^2} dt,
\end{align*}
which, naming 
\[
\Gamma_{T,1}=\tfrac1T\left(2c_{r_0}+k_1c_{r_0}^{p+2}T+\tfrac{L_0c_{r_0}(1+2cr_0^2)T}{2}+K_0c_{r_0}T\right)
\]
and $\Gamma_{T,2}=(\frac{4}{k_0})^{\frac{p}{p+2}}T^{-\frac{2}{p+2}}$, gives us \eqref{ghxxgh}.
\end{proof}

\begin{corollary}\label{cor:itemi}
We have
\begin{align*}
\mathcal{E}_s(T)& \leqslant \mathcal{E}_s(0) + \left|\int_0^T \langle f_s(\tau, v)-f_s(\tau,w), z_t\rangle d\tau\right|+ 2c_{r_0}\int_0^T\left\|\int_{\Omega} K(x,y)z_t(\tau,y)dy\right\|_{L^2}dt.
\end{align*}
\end{corollary}
\begin{proof}
It follows directly from \eqref{inx0} and the fact that $\int_0^T\|z_t\|^2_{L^2}dt \geqslant 0$.
\end{proof}

Our goal for now is to use Theorem \ref{thmplm} to prove that the evolution process $S$ associated with \eqref{ourproblem} is polynomially pullback $\kappa$-dissipative. So, from this point on, we verify that we can use Theorem \ref{thmplm}.

\medskip \noindent \textbf{The pseudometrics $\rho_1$ and $\rho_2$.} \medskip

For $V_1,V_2\in X$, setting $Z(t)=(z(t),z_t(t))=V(t,V_1)-V(t,V_2)$, define the maps $\rho_1,\rho_2\colon X\times X\to \mathbb{R}^+$ as
\begin{equation*}
\rho_1(V_1,V_2)=4c_{r_0}\int_0^T \left\|\int_{\Omega}K(x,y)z_t(t,y) dy\right\|_{L^2} dt,
\end{equation*}
and 
\begin{equation*}\label{def.rho2}
\rho_2(V_1,V_2)=2\Gamma_{T,1}\sup_{t\in [0,T]} \|z(t)\|_{L^2}.
\end{equation*}

Following as in \cite{ZhaoZhong2022} we obtain the following result.

\begin{proposition}\label{rho1rho2}
The functions $\rho_1,\rho_2$ are pseudometrics in $X$, which are precompact in $\overline{B}^X_{r_0}$.
\end{proposition}

\medskip \noindent \textbf{The contractive maps $\psi_1$ and $\psi_2$.} \medskip

As in the definition of $\rho_1$ and $\rho_2$, for $V_1,V_2\in \overline{B}_{r_0}^X$, if $V(t,V_1)=(v(t),v_t(t))$, $V(t,V_2)=(w(t),w_t(t))$, setting $Z(t)=(z(t),z_t(t))=V(t,V_1)-V(t,V_2)$ then $z=v-w$, $z_t=v_t-w_t$, and we define $\psi_1,\psi_2\colon X\times X\to \mathbb{R}^+$ by
\begin{align*}
& \psi_1(V_1,V_2)= 2\left|\int_0^T \langle f_s(\tau,v)-f_s(\tau, w), z_t\rangle d\tau\right|,\\
& \psi_2(V_1,V_2) = \frac{2}{T}\left|\int_0^T\int_t^T \langle f_s(\tau, w)-f_s(\tau, v), w_t-v_t\rangle d\tau dt\right|. 
\end{align*}

Our goal for now is to prove the following result.

\begin{proposition}\label{prop.psi12}
$\psi_1,\psi_2\in \operatorname{contr}(\overline{B}^X_{r_0})$.
\end{proposition}

As we proceeded in previous sections, we will state a sequence of auxiliary lemmas to help us prove Proposition \ref{prop.psi12}. The next one is a convergence result, which uses Alaoglu's Theorem, the inclusions $H^1_0(\Omega)\hookrightarrow \hookrightarrow H^\gamma(\Omega)\hookrightarrow L^2(\Omega)$ and \cite[Corollary 4]{jsimon}.

\begin{lemma}\label{convhs}
Let $\{V_n\}_{n\in\mathbb{N}}\subset \overline{B}^X_{r_0}$ and denote $V(t,V_n)=(v^{(n)}(t),v_t^{(n)}(t))$. For a fixed $\gamma\in (0,1)$, up to a subsequence, we have
\begin{equation*}
\left\{\begin{aligned}
    &(v^{(n)},v_t^{(n)})\overset{\ast}{\rightharpoonup} (v,v_t) \hbox{ in }L^{\infty}(0,T; X),\\
    & v^{(n)}\rightarrow v \hbox{ in } C([0,T]; H^\gamma(\Omega)).
    \end{aligned}\right.
\end{equation*}
\end{lemma}

The next result is a consequence of Proposition \ref{funcFimp}, H\"older's inequality and the continuous inclusions $H^1_0(\Omega)\hookrightarrow L^4(\Omega)$ and $H^\gamma(\Omega)\hookrightarrow L^2(\Omega)$.

\begin{lemma}\label{convhs2}
There exists a constant $C>0$ such that for all $t\geqslant 0$ and $n\in \mathbb{N}$ we have
\[
\left|\int_{\Omega}F_s(t,v^{(n)}(t)) dx-\int_{\Omega} F_s(t,v(t))dx\right|\leqslant C\|v^{(n)}(t)-v(t)\|_{H^\gamma},
\]
and
\[
\left|\int_{\Omega}\frac{\partial F_s}{\partial t}(t,v^{(n)}(t)) dx-\int_{\Omega} \frac{\partial F_s}{\partial t}t,v(t))dx\right|\leqslant C\|v^{(n)}(t)-v(t)\|_{H^\gamma}.
\]
\end{lemma}

From Lemmas \ref{convhs} and \ref{convhs2} we obtain the following result.

\begin{corollary}\label{proppart0}
We have
\[
\int_{\Omega}F_s(t, v^{(n)}(t)) dx \stackrel{n\to \infty}{\longrightarrow} \int_{\Omega} F_s(t, v(t)) dx,
\quad \hbox{ and } \quad 
\int_{\Omega}\frac{\partial F_s}{\partial t}(t, v^{(n)}(t)) dx \stackrel{n\to \infty}{\longrightarrow} \int_{\Omega} \frac{\partial F_s}{\partial t}(t, v(t)) dx,
\]
uniformly for $t\in [0,T]$.
\end{corollary}

\begin{lemma}\label{mnbvcxzas}
For each $t\in [0,T]$ and $\xi \in L^1(0, T; H^3(\Omega)\cap H^1_0(\Omega))$ it holds that
\[
\int_t^T \langle f_s(\tau, v^{(n)}(\tau))-f_s(\tau, v(\tau)), \xi(\tau)\rangle d\tau \stackrel{n\to \infty}{\longrightarrow} 0.
\]
\end{lemma}
\begin{proof}
Since $H^3(\Omega)\hookrightarrow L^\infty(\Omega)$, we have $L^1(\Omega)\hookrightarrow (L^{\infty}(\Omega))^{\star}\hookrightarrow H^{-3}(\Omega)$. Using Proposition \ref{funcFimp} we obtain
\begin{align*}
&\|f_s(t, v^{(n)}(t))-f_s(t, v(t))\|_{H^{-3}} \leqslant C\|f_s(t, v^{(n)}(t))-f_s(t, v(t))\|_{L^1}\\
&\quad\leqslant C \int_{\Omega}(1+|v^{(n)}(t)|^2+|v(t)|^2)|v^{(n)}(t)-v(t)| dx\\
&\quad\leqslant C \left(\int_{\Omega}(1+|v^{(n)}(t)|^4+|v(t)|^4)dx\right)^\frac12\|v^{(n)}(t)-v(t)\|_{L^2}\\
&\quad \leqslant C (1+\|v^{(n)}(t)\|_{L^4}^2 + \|v(t)\|_{L^4}^2)\|v^{(n)}(t)-v(t)\|_{L^2}\\
&\quad \leqslant C (1+\|v^{(n)}(t)\|_{H^1_0}^2 + \|v(t)\|_{H^1_0}^2)\|v^{(n)}(t)-v(t)\|_{H^\gamma} \leqslant C\|v^{(n)}(t)-v(t)\|_{H^\gamma},
\end{align*}
since $H^\gamma(\Omega)\hookrightarrow L^2(\Omega)$, $H^1_0(\Omega)\hookrightarrow L^4(\Omega)$ and the fact that $\|v^{(n)}(t)\|_{H^1_0},\|v(t)\|_{H^1_0}\leqslant c_{r_0}$ for all $t\geqslant 0$.

From Lemma \ref{convhs} we conclude that
\[
    \sup_{t\in [0,T]}\|(f_s(t, w^{(n)}(t))-f_s(t, w(t)))\|_{H^{-3}}\stackrel{n\to \infty}{\longrightarrow} 0.
\]
Hence, if $\xi \in L^1(0, T, H^3(\Omega)\cap H^1_0(\Omega))$ we have
\begin{align*}
&\left|\int_t^T \langle f_s(\tau, v^{(n)}(\tau))-f_s(\tau, v(\tau)), \xi(\tau) \rangle d\tau\right| \leqslant \sup_{t\in [0,T]}\|f_s(\tau,v^{(n)}(\tau))-f_s(\tau,v(\tau))\|_{H^{-3}}\int_0^T\|\xi(\tau)\|_{H^3}d\tau,
\end{align*}
and the result is proven.  
\end{proof}

\begin{corollary}\label{ckkk}
For each $t\in [0,T]$ we have
\[
f_s(\tau, v^{(n)}(\tau))\stackrel{\ast}{\rightharpoonup} f_s(\tau, v(\tau)) \hbox{ in } L^{\infty}(t, T; L^2(\Omega)) \hbox{ as } n\to \infty.
\]
\end{corollary}
\begin{proof}
    Since $L^1(t,T;H^3(\Omega)\cap H^1_0(\Omega))\hookrightarrow L^1(t,T;L^2(\Omega)$ with dense inclusion (see \cite[Lemma 1.2.19]{Tuomas}, for instance), the result follows.
\end{proof}

These convergences imply the following.

\begin{lemma}\label{proppart1}
For each $t\in [0,T]$ and $n\in \mathbb{N}$ we have
\[
\int_t^T\langle f_s(\tau, v^{(n)}(\tau)),v_t^{(m)}(\tau)\rangle d\tau \stackrel{m\to \infty}{\longrightarrow} \int_t^T \langle f_s(\tau, v^{(n)}(\tau)),v_t(\tau)\rangle d\tau,
\]
and
\[
\int_t^T\langle f_s(\tau, v^{(m)}(\tau)),v_t^{(n)}(\tau) \rangle d\tau \stackrel{m\to \infty}{\longrightarrow} \int_t^T \langle f_s(\tau, v(\tau)),v_t^{(n)}(\tau)\rangle d\tau.
\]
\end{lemma}

Using  Corollary \ref{ckkk}, we obtain the next result.

\begin{lemma}\label{proppart2}
For each $t\in [0,T]$ we have
\[
\int_t^T \langle f_s(\tau, v^{(n)}(\tau)),v_t(\tau)\rangle d\tau \stackrel{n\to \infty}{\longrightarrow} \int_t^T \langle f_s(\tau, v(\tau)),v_t(\tau)\rangle d\tau,
\]
and
\[
\int_t^T \langle f_s(\tau, v(\tau)),v^{(n)}_t(\tau)\rangle d\tau \stackrel{n\to \infty}{\longrightarrow} \int_t^T \langle f_s(\tau, v(\tau)),v_t(\tau)\rangle d\tau.
\]
\end{lemma}

We can now state the following.

\begin{lemma}\label{propt2023x}
For each $t\in [0,T]$ 
\begin{align*}
&\lim_{n\to \infty}\lim_{m\to \infty}\int_t^T \langle f_s(\tau, v^{(n)}(\tau)),v_t^{(m)}(\tau)\rangle d\tau\\
&\quad=\int_{\Omega}F_s(T,v(T))dx-\int_{\Omega}F_s(t,v(t))dx-\int_{\Omega}\int_t^T \frac{\partial}{\partial t}F_s(\tau, v(\tau)) d\tau dx,
\end{align*}
and
\begin{align*}
&\lim_{n\to \infty}\lim_{m\to \infty}\int_t^T \langle f_s(\tau, v^{(m)}(\tau)),v_t^{(n)}(\tau)\rangle d\tau\\
&\quad=\int_{\Omega}F_s(T,v(T))dx-\int_{\Omega}F_s(t,v(t))dx-\int_{\Omega}\int_t^T \frac{\partial}{\partial t}F_s(\tau, v(\tau)) d\tau dx.
\end{align*}
\end{lemma}
\begin{proof}
    The result is a consequence of Fubini's Theorem and Lemmas \ref{proppart1} and \ref{proppart2}.
\end{proof}

Using Fubini's Theorem and Lebesgue's Dominated Convergence Theorem, we obtain the properties that will give us the contractiveness of the maps $\psi_1$ and $\psi_2$.

\begin{corollary}\label{corphi1}
We have
\[
\lim_{n\to \infty}\lim_{m\to \infty}\int_0^T\langle f_s(\tau,v^{(n)}(\tau))-f_s(\tau,v^{(m)}(\tau)),v_t^{(n)}(\tau)-v_t^{(m)}(\tau)\rangle d\tau =0,
\]
and for each $t\in [0,T]$ we have
\[
\lim_{n\to \infty}\lim_{m\to \infty} \int_0^T\int_t^T\langle f_s(\tau,v^{(n)}(\tau))-f_s(\tau,v^{(m)}(\tau)),v_t^{(n)}(\tau)-v_t^{(m)}(\tau)\rangle d\tau dt = 0.
\]
\end{corollary}

With Corollary \ref{corphi1}, the proof of Proposition \ref{prop.psi12} is trivial.

\bigskip \noindent \textbf{The generalized polynomial pullback attractor: the proof of Theorem \ref{App:PolAtt}.} \medskip

\begin{proof}
Consider the functions $g_1, g_2\colon \mathbb{R^+}\times \mathbb{R}^+$ given by $g_1(\alpha, \beta)= \alpha$ and $g_2(\alpha, \beta)=\alpha+\beta$. It is clear that $g_1, g_2$ are non-decreasing with respect to each variable, $g_1(0,0)=g_2(0,0)=0$ and they are continuous at $(0,0)$.  Also , we observe that for any $s\in \mathbb{R}$ we have 
\begin{align*}    
2\mathcal{E}_s(\tau)&=\|Z(\tau)\|^2_X = \|V(\tau,V_1)-V(\tau,V_2)\|^2_X = \|S(\tau+s,S)V_1-S(\tau+s,s)V_2\|^2_X,
\end{align*}
From Corollary \ref{cor:itemi} it follows that
\[
2\mathcal{E}_s(T) \leqslant 2\mathcal{E}_s(0) + \psi_1(V_1,V_2) + \rho_1(V_1,V_2) =  2\mathcal{E}_s(0) + \psi_1(V_1,V_2) + g_1(\rho_1(V_1,V_2),\rho_2(V_1,V_2)),
\]
and from Proposition \ref{prop.ghxxgh} it follows that 
\begin{align*}
2\mathcal{E}_s(T) & \leqslant 2^{\frac{2}{p+2}}\Gamma_{T,2}\Big(2\mathcal{E}_s(0) - 2\mathcal{E}_s(T) + \rho_1(V_1,V_2) + \psi_1(V_1,V_2)\Big)^{\frac{2}{p+2}} \\
& \qquad + \rho_1(V_1,V_2) + \rho_2(V_1,V_2) + \psi_2(V_1,V_2) \\
& \leqslant 2^{\frac{2}{p+2}}\Gamma_{T,2}\Big(2\mathcal{E}_s(0) - 2\mathcal{E}_s(T) + g_1(\rho_1(V_1,V_2),\rho_2(V_1,V_2)) \psi_1(V_1,V_2)\Big)^{\frac{2}{p+2}} \\
& \qquad  + g_2(\rho_1(V_1,V_2), \rho_2(V_1,V_2)) + \psi_2(V_1,V_2).
\end{align*}
Finally, Proposition \ref{dcdi} gives the last hypotheses for us to apply Theorem \ref{corMain} and guarantee that the evolution process $S$, associated with \eqref{ourproblem}, is $\varphi$-pullback $\kappa$-dissipative, with decay function $\varphi(t)=t^{-\frac1p}$, and that $S$ has a bounded generalized $\varphi$-pullback attractor $\hat{M}$. Furthermore, from Theorem \ref{theo:GenimpliesPA}, $S$ has a pullback attractor $\hat{A}$, with $\hat{A}\subset \hat{M}$. 
\end{proof}

\appendix

\section{Auxiliary results}\label{appendix}
The following appendix contains technical results that are complementary to the main theory presented in this work.

First we note that from the Mean Value Theorem applied to the function $f(t)=t^{1-\frac{1}{\beta}}$, the next lemma is obvious.  

\begin{lemma}\label{lema.A2}
Let $\beta\in(0,1)$. If $b>a>0$, there exists $\theta\in(0,1)$ such that 
\[
b^{1-\frac{1}{\beta}}-a^{1-\frac{1}{\beta}}=\left(1-\tfrac{1}{\beta}\right)\left[\theta a+(1-\theta)b\right]^{-\frac{1}{\beta}}(b-a).
\]
\end{lemma}

With this we prove the following.

\begin{proposition}\label{functionsuv}
Consider the function $u\colon \mathbb{R}^+ \to \mathbb{R}^+$ given by $u(t)=(3C)^{-1/{\beta}}t^{1/{\beta}}+t$, where $C>0$ and $0<\beta<1$ are constants. Let $v\colon \mathbb{R}^+\to \mathbb{R}^+$ be its inverse function and fix a real number $t_0\geqslant  0$. Then the sequence $\left\{t_n\right\}$ defined by $t_n=v^n(t_0)$ for each $n\in \mathbb{N}$ satisfies
\begin{enumerate}[label={(\roman*)}]
\item \label{z1} $\left\{t_n\right\}$ is non-increasing;
\item \label{z2} $t_n-t_{n+1}=(3C)^{-1/{\beta}}(t_{n+1})^{1/{\beta}}$;
\item \label{z3} $t_n\to 0$ when $n\to \infty$;
\item \label{z4} there exists $n_0\in\mathbb{N}$ such that for all  $n\geqslant  n_0$ we have
\[
(t_n)^{1-1/{\beta}}\geqslant  \left(\tfrac{1}{{\beta}}-1\right)(1+3C)^{-1/{\beta}}+(t_{n-1})^{1-1/{\beta}};
\]
\item \label{z5} with $n_0$ as in \ref{z4} if $n\geqslant n_0$ and $k\in \mathbb{N}$ we have 
\[
(t_{n_0+k})^{1-1/{\beta}}\geqslant  k\left(\tfrac{1}{{\beta}}-1\right)(1+3C)^{-1/{\beta}}+(t_{n_0})^{1-1/{\beta}};
\]
\item \label{z6} with $n_0$ as in \ref{z4} if $n\geqslant n_0$ we have
\begin{align*} 
t_n=v^n(t_0)\leqslant  \left[(n-n_0)\left(\tfrac{1}{\beta}-1\right)(1+3C)^{-\frac{1}{\beta}}+(t_0)^{1-\frac{1}{\beta}}\right]^{\frac{\beta}{\beta-1}}.
\end{align*}
\end{enumerate}

\begin{proof}

(i) Since $u'(t)=\frac{1}{\beta (3C)^{1/{\beta}}}t^{\frac{1}{\beta}-1}+1>0$ for all $t\geqslant  0$, $u$ is an increasing function and, consequently, its inverse function $v$ is also increasing. Then, since $u(t)\geqslant  t$ for all $t\geqslant  0$, we have $t=v(u(t))\geqslant  v(t)$ for all $t\geqslant  0$. Now, it is clear that $t_0\geqslant  v(t_0)=t_1\geqslant  v(t_1)=t_2\geqslant  v(t_2)=t_3\geqslant  \cdots$

\par \medskip (ii) $t_n=u(t_{n+1})=(3C)^{-1/{\beta}}\left(t_{n+1}\right)^{1/{\beta}}+t_{n+1}$.

\par \medskip (iii) Since $\left\{t_n\right\}$ is non-increasing we have $t_n\to \alpha$. Assume that $\alpha>0$. We have $t_n-t_{n+1}\to 0$ and, from \ref{z2}, we have $t_n-t_{n+1}\to (3C)^{-1/{\beta}}\alpha^{1/{\beta}}>0$, which is a contradiction.

\par \medskip (iv) Since $\left\{t_n\right\}$ is non-increasing and $t_n\to 0$, there exists $n_0\in\mathbb{N}$ such that $t_{n-1}-t_n\in (0,1)$ for all $n\geqslant  n_0$. Furthermore, from \ref{z2}, we have $t_{n}=(3C)\left(t_{n-1}-t_n\right)^{\beta}$ for all $n\geqslant 1$. Now, if $n\geqslant  n_0$ we have
\begin{align*}
t_{n-1}&=t_n+\left(t_{n-1}-t_n\right)=(3C)\left(t_{n-1}-t_n\right)^{\beta}+\left(t_{n-1}-t_n\right)\\
& \leqslant  3C\left(t_{n-1}-t_n\right)^{\beta}+\left(t_{n-1}-t_n\right)^{\beta}=(1+3C)\left(t_{n-1}-t_n\right)^{\beta}.
\end{align*}
Then, using 
Lemma \ref{lema.A2}, for $n\geqslant n_0$ we have
\begin{align*}
(t_n)^{1-\frac{1}{\beta}}-(t_{n-1})^{1-\tfrac{1}{\beta}}&=\underbrace{\left(1-\tfrac{1}{\beta}\right)}_{<0}\underbrace{\left[\theta t_n+(1-\theta)t_{n-1}\right]^{-\frac{1}{\beta}}}_{\geqslant  (t_{n-1})^{-\frac{1}{\beta}}}\underbrace{\left(t_n-t_{n-1}\right)}_{<0}\\
&\geqslant  \left(1-\tfrac{1}{\beta}\right)(t_{n-1})^{-\frac{1}{\beta}}(t_n-t_{n-1})\geqslant  \left(\tfrac{1}{\beta}-1\right)(1+3C)^{-\frac{1}{\beta}},
\end{align*}
concluding the proof of item \ref{z4}.

\par \medskip (v) The cases $k=0$ and $k=1$ are trivial from \ref{z4}. Suppose that \ref{z5} is valid for a $k\in\mathbb{N}$. For $k+1$ we have
\begin{align*}
\left(t_{n_0+(k+1)}\right)^{1-\tfrac{1}{\beta}}&\geqslant  \left(\tfrac{1}{\beta}-1\right)(1+3C)^{-\frac{1}{\beta}}+\left(t_{n_0+k}\right)^{1-\frac{1}{\beta}}\\
&\geqslant  \left(\tfrac{1}{\beta}-1\right)(1+3C)^{-\frac{1}{\beta}}+k\left(\tfrac{1}{\beta}-1\right)(1+3C)^{-\frac{1}{\beta}}+(t_{n_0})^{1-\frac{1}{\beta}}\\
&=(k+1)\left(\tfrac{1}{\beta}-1\right)(1+3C)^{-\frac{1}{\beta}}+(t_{n_0})^{1-\frac{1}{\beta}}.
\end{align*}

\par \medskip (vi) Since $\frac{\beta}{\beta-1}<0$, it is clear from \ref{z5} that if $n\geqslant  n_0$ it follows that
\begin{align*}
t_n=v^n(t_0)&\leqslant  \left[(n-n_0)\left(\tfrac{1}{\beta}-1\right)(1+3C)^{-\frac{1}{\beta}}+(t_{n_0})^{1-\frac{1}{\beta}}\right]^{\frac{\beta}{\beta-1}}\\
&\leqslant  \left[(n-n_0)\left(\tfrac{1}{\beta}-1\right)(1+3C)^{-\frac{1}{\beta}}+(t_0)^{1-\frac{1}{\beta}}\right]^{\frac{\beta}{\beta-1}}.
\end{align*}
\end{proof}
\end{proposition}

\begin{proposition}\label{upLipschitz}
    If $H$ is a real Hilbert space with inner product $\langle \cdot, \cdot\rangle$ and associate norm $\|\cdot\|$, then for $p\geqslant 0$ and $u,v\in \overline{B}_R^H$ we have
    \begin{equation}\label{prop.ineqcp}
        \left\|\|u\|^pu-\|v\|^pv\right\|\leqslant (p+1)R^p\|u-v\|.
    \end{equation}
\begin{proof}
    When $p=0$, \eqref{prop.ineqcp} is a trivial equality. When $v=0$ then \eqref{prop.ineqcp} is also straightforward. Thus, we consider $p>0$ and we can assume, without loss of generality, that $0<\|v\|\leqslant \|u\|$. Observe that \eqref{prop.ineqcp} is equivalent to
    \begin{equation}\label{prop.ineqcp2}       \|u\|^{2p+2}-2\|u\|^p\|v\|^p\langle u,v\rangle + \|v\|^{2p+2}\leqslant (p+1)^2R^{2p}(\|u\|^2-2\langle u,v\rangle + \|v\|^2).
    \end{equation}
    Also note that $\|u\|^2-2\langle u,v\rangle + \|v\|^2=0$ iff $\|u-v\|=0$ iff $u=v$, in which case the inequality is also trivial. Thus we can also assume that $u\neq v$. 
    Setting $x=\|u\|$, $y=\|v\|$ and $z=\langle u,v\rangle$, we have $0<y\leqslant x \leqslant R$, $-xy\leqslant z \leqslant xy$, $x^2-2z+y^2\neq 0$ and \eqref{prop.ineqcp2} becomes 
    \begin{equation}\label{prop.ineqcp3}
        x^{2p+2}-2x^py^pz+y^{2p+2}\leqslant (p+1)^2R^{2p}(x^2-2z+y^2).
    \end{equation}

\medskip \noindent \textsc{Case I.} $-xy\leqslant z \leqslant 0$. 

In this case, since $z\leqslant 0$, $x^p \leqslant (p+1)R^p$, and $y^p\leqslant (p+1)R^p$, we have
\begin{align*}
    &x^{2p+2}-2x^py^pz+y^{2p+2}-(p+1)^2R^{2p}(x^2-2z+y^2)\\
    &\quad = x^2[x^p-(p+1)R^{p}][x^p+(p+1)R^{p}]+y^2[y^p-(p+1)R^p][y^p+(p+1)R^p]\\
    &\qquad -2z(x^py^p-(p+1)^2R^{2p})\leqslant 0,
\end{align*}
which proves \eqref{prop.ineqcp3}.

\medskip \noindent \textsc{Case II.} $0< z \leqslant xy$.

Define 
\begin{equation*}    
f(x,y,z)=\frac{x^{2p+2}-2x^py^pz+y^{2p+2}}{x^2-2z+y^2}.
\end{equation*}
We have 
\begin{equation*}
    \frac{\partial f}{\partial z}(x,y,z)=\frac{(-2x^py^p)(x^2-2z+y^2)+2(x^{2p+2}-2x^py^pz+y^{2p+2})}{(x^2-2z+y^2)^2},
\end{equation*}
and since 
\begin{align*}
    &(-2x^py^p)(x^2-2z+y^2)+2(x^{2p+2}-2x^py^pz+y^{2p+2})\\
    &\quad =2x^{p+2}(x^p-y^p)-2y^{p+2}(x^p-y^p)=2(x^p-y^p)(x^{p+2}-y^{p+2})\geqslant 0,
\end{align*}
it follows that $\frac{\partial f}{\partial z}(x,y,z)\geqslant 0$. Thus, for each pair $(x,y)$ satisfying the conditions mentioned above, the function $(0,xy]\ni z \mapsto f(x,y,z)$ is increasing and, consequently, $f(x,y,z)\leqslant f(x,y,xy)$. Set 
\begin{equation*}
g(x,y):=f(x,y,xy)=\frac{x^{2p+2}-2x^{p+1}y^{p+1}+y^{2p+2}}{x^2-2xy+y^2}=\frac{(x^{p+1}-y^{p+1})^2}{(x-y)^2}.
\end{equation*}

A direct application of the Mean Value Theorem to the function $h(t)=t^{p+1}$ shows that there exists $\xi\in (y,x)$ such that
\begin{equation*}
    x^{p+1}-y^{p+1}=(p+1)\xi^p(x-y)\leqslant (p+1)x^p(x-y)\leqslant (p+1)R^p(x-y),
\end{equation*}
which implies $f(x,y,z)\leqslant g(x,y)\leqslant (p+1)^2R^{2p}$, and completes the proof.
\end{proof}    
\end{proposition}

The next result is taken from \cite[Lemma 2.2]{zhao2020global}, and we present a slightly adapted proof.

\begin{proposition}\label{prop.Cp}
If $H$ is a real Hilbert space with inner product $\langle \cdot, \cdot\rangle$ and associate norm $\|\cdot\|$, then for $p\geqslant 0$ and $x,y\in H$ we have
\begin{equation}\label{prop.Cp.eq}
\langle \|x\|^px-\|y\|^py,x-y\rangle \geqslant 2^{-p} \|x-y\|^{p+2}.
\end{equation}
\end{proposition}
\begin{proof}
When $p=0$, the inequality is trivial and $C_0=1$. Hence we assume $p>0$. Clearly \eqref{prop.Cp.eq} is trivial when $x=0$. We can assume, without loss of generality, that $\|y\|\leqslant \|x\|$ and $x\neq 0$. Note that \eqref{prop.Cp.eq} is equivalent to
\[
\|x\|^{p+2} + \|y\|^{p+2} - (\|x\|^p +\|y\|^p) \langle x,y\rangle \geqslant 2^{-p}(\|x\|^2 + \|y\|^2 - 2\langle x,y\rangle)^{\frac{p+2}{2}}.
\]
Dividing both sides by $\|x\|^{p+2}$ we obtain
\[
1+\frac{\|y\|^{p+2}}{\|x\|^{p+2}} - \left(1+\frac{\|y\|^p}{\|x\|^p}\right)\frac{\|y\|}{\|x\|}\frac{\langle x,y\rangle}{\|x\|\|y\|} \geqslant 2^{-p} \left( 1+\frac{\|y\|^2}{\|x\|^2}-2\frac{\|y\|}{\|x\|}\frac{\langle x,y\rangle}{\|x\|\|y\|}\right)^{\frac{p+2}{2}}
\]
Setting $t=\frac{\|y\|}{\|x\|}\in [0,1]$ and $s=\frac{\langle x,y\rangle}{\|x\|\|y\|} \in [-1,1]$ we obtain
\begin{equation}\label{prop.Cp.eq.2}
1+t^{p+2}-(t+t^{p+1})s \geqslant 2^{-p} (1+t^2-2ts)^{\frac{p+2}{2}}.
\end{equation}
Since $1+t^2-2ts = (t-s)^2 + 1-s^2 = 0$ if and only if $t=s=1$, we can prove \eqref{prop.Cp.eq.2} for $t\in [0,1)$ and $s\in [-1,1)$, where $1+t^2-2ts\neq 0$. Define $h\colon [0,1)\times [-1,1)\to \mathbb{R}$ by
\[
h(t,s) = \frac{1+t^{p+2}-(t+t^{p+1})s}{(1+t^2-2ts)^{\frac{p+2}{2}}}.
\]
We have
\[
\frac{\partial h}{\partial s}(t,s) = \frac{t[(p+2)(1+t^{p+2})-(1+t^p)(1+t^2) - p(t+t^{p+1})s]}{(1+t^2-2ts)^{\frac{p+4}{2}}},
\]
and, for $s<1$, 
\begin{align*}
(p+2)(1+t^{p+2}) & -(1+t^p)(1+t^2) - p(t+t^{p+1})s  \\
& \geqslant (p+2)(1+t^{p+2})-(1+t^p)(1+t^2) - p(t+t^{p+1}).
\end{align*}
For all $t\in [0,1]$ note that
\begin{align*}
(p+2)& (1+t^{p+2})-(1+t^p)(1+t^2) - p(t+t^{p+1})\\
& = p(1-t^{p+1})(1-t) + (1-t^p)(1-t^2) \geqslant 0
\end{align*}
thus $\frac{\partial h}{\partial s}(t,s)\geqslant 0$ for $s\in [-1,1)$ and $t\in [0,1)$. Hence, for each fixed $t\in [0,1)$, $s\mapsto h(t,s)$ is increasing, which implies that $h(t,s)\geqslant h(t,-1)$ for each $s\in [-1,1)$ and $t\in [0,1)$. Set
\[
g(t) = h(t,-1) = \frac{t^{p+2}+t^{p+1}+t+1}{(1+t)^{p+2}},
\]
and note that for $t\in [0,1]$ we have
\[
g'(t) = \frac{(p+1)(t+1)(t^p-1)}{(1+t)^{p+3}} \leqslant 0.
\]
Thus $g$ is decreasing and hence
\[
h(t,s) \geqslant h(t,-1) = g(t) \geqslant g(1) = 2^{-p},
\]
which proves \eqref{prop.Cp.eq.2} and completes the result.
\end{proof}

\medskip \textbf{Data Availability.} Data sharing is not applicable to this article as no new data were created or analyzed in
this study.

\section*{Declarations}

\textbf{Competing Interests.} The authors have not disclosed any competing interests.

\bibliographystyle{abbrv}

\end{document}